\documentclass[11pt,reqno]{amsart}

\usepackage{amssymb}
\usepackage[margin=3cm]{geometry}

\usepackage{hyperref}
\usepackage{xcolor}
\definecolor{darkgreen}{rgb}{0,0.45,0}
\hypersetup{colorlinks,urlcolor=blue,citecolor=darkgreen,linkcolor=blue,linktocpage=false}
\usepackage[abbrev]{amsrefs}
\usepackage[all]{xy}
\newdir{ >}{{}*!/-8pt/@{>}} %
\newcommand{\cxymatrix}[1]{\vcenter{\xymatrix{#1}}} %
\SilentMatrices

\usepackage{ifpdf}
\ifpdf
\else
\usepackage{breakurl}
\fi

\usepackage{mathtools} %

\numberwithin{equation}{section}

\theoremstyle{plain}
\newtheorem{thm}{Theorem}[section]
\newtheorem*{thm*}{Theorem} %
\newtheorem{prop}[thm]{Proposition}
\newtheorem{lem}[thm]{Lemma}
\newtheorem{cor}[thm]{Corollary}
\newtheorem*{cor*}{Corollary} %

\theoremstyle{definition}
\newtheorem{assum}[thm]{Assumption}
\newtheorem{condition}[thm]{Condition}
\newtheorem{convention}[thm]{Convention}
\newtheorem{defn}[thm]{Definition}
\newtheorem{ex}[thm]{Example}
\newtheorem{nota}[thm]{Notation}

\theoremstyle{remark}

\newtheorem{rem}[thm]{Remark}

\newcommand{\bigcol}{3.3pc}

\newcommand{\Def}{\textbf}

\newcommand{\F}{\mathbb{F}}
\newcommand{\Z}{\mathbb{Z}}

\newcommand{\cat}[1]{\mathbf{\mathcal{#1}}} %

\newcommand{\cT}{\cat{T}}
\newcommand{\cTo}{\cat{T}^{\opp}}

\newcommand{\al}{\alpha}
\newcommand{\be}{\beta}
\newcommand{\de}{\delta}

\newcommand{\ga}{\gamma}
\newcommand{\io}{\iota}
\newcommand{\ka}{\kappa}
\newcommand{\la}{\lambda}
\newcommand{\Om}{\Omega}
\newcommand{\phy}{\varphi}
\newcommand{\Si}{\Sigma}
\newcommand{\si}{\sigma}

\newcommand{\te}{\theta}

\newcommand{\circar}{\ar|(.48)*\cir<1.7pt>{}} %
\newcommand{\circto}{{\ooalign{$\longrightarrow$\cr\hidewidth$\circ$\hidewidth\cr}}}
\newcommand{\inj}{\hookrightarrow}
\newcommand{\ral}{\xrightarrow} %
\newcommand{\Ra}{\Rightarrow} %
\newcommand{\surj}{\twoheadrightarrow}
\newcommand{\lra}{\longrightarrow}

\newcommand{\bu}{\bullet}
\newcommand{\op}{\oplus}

\newcommand{\sm}{\wedge} %

\newcommand{\lan}{\left\langle}
\newcommand{\ran}{\right\rangle}
\newcommand{\tild}{\widetilde}

\newcommand{\Ab}{\mathrm{Ab}}

\newcommand{\dgMod}[1]{\mathrm{dgMod}_{#1}}

\newcommand{\Mod}[1]{\mathrm{Mod}_{#1}}
\newcommand{\StMod}{\mathrm{StMod}}

\newcommand{\abs}[1]{\lvert #1 \rvert}

\DeclareMathOperator{\Ext}{Ext}
\DeclareMathOperator{\Hom}{Hom}
\newcommand{\id}{\mathrm{id}}
\newcommand{\opp}{\mathrm{op}}

\newcommand{\steen}{\mathcal{A}}

\newcommand{\TF}{\mathrm{T}}
\newcommand{\ff}{\!\mathit{ff}} %
\newcommand{\fc}{\!\mathit{fc}} %
\newcommand{\cc}{\!\mathit{cc}} %
\newcommand{\Ship}{\!\mathit{SS}} %

\newcommand{\fix}{\overset{!\hspace*{1.5pt}}{,\rule{0pt}{3pt}}}

\begin{document}

\title[Higher Toda brackets and the Adams spectral sequence]
      {Higher Toda brackets and the Adams spectral sequence in triangulated categories} 
\date{\today}

\author{J. Daniel Christensen}
\address{Department of Mathematics\\
University of Western Ontario\\
London, Ontario, N6A 5B7\\
Canada}             
\email{jdc@uwo.ca}

\author{Martin Frankland}             
\address{Universit\"at Osnabr\"uck\\
Institut f\"ur Mathematik\\
Albrechtstr. 28a\\
49076 Osnabr\"uck\\
Germany}             
\email{martin.frankland@uni-osnabrueck.de}

\begin{abstract}
The Adams spectral sequence is available in any triangulated category equipped with a projective or injective class. Higher Toda brackets can also be defined in a triangulated category, as observed by B. Shipley based on J. Cohen's approach for spectra. We provide a family of definitions of higher Toda brackets, show that they are equivalent to Shipley's, and show that they are self-dual. Our main result is that the Adams differential $d_r$ in any Adams spectral sequence can be expressed as an $(r+1)$-fold Toda bracket and as an $r^{\text{th}}$ order cohomology operation.  We also show how the result simplifies under a sparseness assumption, discuss several examples, and give an elementary proof of a result of Heller, which implies that the three-fold Toda brackets in principle determine the higher Toda brackets.
\end{abstract}



\maketitle

\tableofcontents

\section{Introduction}

The Adams spectral sequence is an important tool in stable homotopy theory. Given finite spectra $X$ and $Y$, the classical Adams spectral sequence is
\[
E_2^{s,t} = \Ext_{\steen}^{s,t}\left( H^*Y, H^*X \right) \Ra [\Si^{t-s}X, Y^{\wedge}_p] ,
\]
where $H^*X := H^*(X;\F_p)$ denotes mod $p$ cohomology and $\steen = H\F_p^* H\F_p$ denotes the mod $p$ Steenrod algebra.
Determining the differentials in the Adams spectral sequence generally requires 
a combination of techniques and much ingenuity.
The approach that provides a basis for our work is found in~\cite{Maunder64},
where Maunder showed that the differential $d_r$ in this spectral sequence is
determined by $r^{\text{th}}$ order cohomology operations, which we now review.

A primary cohomology operation in this context is simply an element of the Steenrod algebra,
and it is immediate from the construction of the Adams spectral sequence 
that the differential $d_1$ is given by primary cohomology operations.
A secondary cohomology operation corresponds to a relation among primary operations,
and is partially defined and multi-valued: it is defined on the kernel of
a primary operation and takes values in the cokernel of another primary operation.
Tertiary operations correspond to relations between relations, and have correspondingly
more complicated domains and codomains.
The pattern continues for higher operations.
Using that cohomology classes are representable,
secondary cohomology operations can also be expressed using $3$-fold Toda brackets
involving the cohomology class and two operations whose composite is null.
However, what one obtains in general is a \emph{subset} of the Toda bracket with less indeterminacy.
This observation will be the key to our generalization of Maunder's result to 
other Adams spectral sequences in other categories.

The starting point of this paper is the following observation.
On the one hand, the Adams spectral sequence can be constructed in any 
triangulated category equipped with a projective class or an injective class~\cite{Christensen98}.
For example, the classical Adams spectral sequence is constructed in the stable homotopy category
with the injective class consisting of retracts of products $\prod_i \Si^{n_i} H\F_p$.
On the other hand, higher Toda brackets can also be defined in an arbitrary triangulated category.
This was done by Shipley in~\cite{Shipley02}, based on Cohen's construction for spaces and spectra~\cite{Cohen68},
and was studied further in~\cite{Sagave08}.
The goal of this paper is to describe precisely how the Adams $d_r$ can be described as
a particular subset of an $(r+1)$-fold Toda bracket which can be viewed as an $r^{\text{th}}$ order cohomology operation,
all in the context of a general triangulated category.

Triangulated categories arise throughout mathematics,
so our work applies in various situations.
As an example, we give calculations involving the Adams spectral sequence in 
the stable module category of a group algebra.
Even in stable homotopy theory, there are a variety of Adams spectral sequences,
such as the Adams--Novikov 
spectral sequence or the motivic Adams spectral sequence, and our results apply to all of them.
Moreover, by working with minimal structure, our approach gains a certain elegance.

\subsection*{Organization and main results}

In Section~\ref{se:AdamsSS}, we review the construction of the Adams spectral sequence in a triangulated category equipped with a projective class or an injective class. In Section~\ref{se:3-fold-Toda-brackets}, we review the construction of $3$-fold Toda brackets in a triangulated category and some of their basic properties. Section~\ref{se:AdamsD2} describes how the Adams $d_2$ is given by $3$-fold Toda brackets. This section serves as a warm-up for Section~\ref{se:AdamsDr}.

In Section~\ref{se:HigherBrackets}, we recall the construction of higher Toda brackets in a triangulated category via filtered objects. We provide a family of alternate constructions, and prove that they are all equivalent. The main result is Theorem~\ref{th:n-fold}, which says roughly the following.

\begin{thm*}
There is an inductive way of computing an $n$-fold Toda bracket $\lan f_n, \ldots, f_1 \ran \subseteq \cat{T}(\Si^{n-2} X_0, X_n)$, where the inductive step picks three consecutive maps and reduces the length by one. The $(n-2)!$ ways of doing this yield the same subset, up to an explicit sign.   
\end{thm*}

As a byproduct, we obtain Corollary~\ref{co:SelfDual}, which would be tricky to prove directly from the filtered object definition.

\begin{cor*}
Toda brackets are self-dual up to suspension: $\lan f_n, \ldots, f_1 \ran \subseteq \cat{T}(\Si^{n-2} X_0, X_n)$ corresponds to the Toda bracket computed in the opposite category
\[
\lan f_1, \ldots, f_n \ran \subseteq \cat{T}^{\opp}\left( \Si^{-(n-2)} X_n, X_0 \right) = \cat{T}(X_0, \Si^{-(n-2)} X_n).
\] 
\end{cor*}

Section~\ref{se:AdamsDr} establishes how the Adams $d_r$ is given by $(r+1)$-fold Toda brackets. Our main results are Theorems~\ref{th:d1pdex} and \ref{th:AdamsDrCohomOp}, which say roughly the following.

\begin{thm*}
Let $[x] \in E_r^{s,t}$ be a class in the $E_r$ term of the Adams spectral sequence. %
As subsets of $E_1^{s+r,t+r-1}$, we have
\begin{align*}
d_r [x] &= \lan \Si^{r-1} d_1 , \ldots , \Si^2 d_1, \Si d_1 , \Si p_{s+1} ,  \de_s x \ran \\
&= \lan \Si^{r-1} d_1 \fix \ldots \fix \Si d_1 \fix d_1 , x \ran .
\end{align*}
\end{thm*}

Here, $d_1$, $p_{s+1}$, and $\de_s$ are maps appearing in the Adams resolution of $Y$, where each $d_1$ is a primary cohomology operation. The first expression for $d_r [x]$ is an $(r+1)$-fold Toda bracket. The second expression denotes an appropriate subset of the bracket $\lan \Si^{r-1} d_1, \ldots, \Si d_1, d_1 , x \ran$ with some choices dictated by the Adams resolution of $Y$. This description exhibits $d_r [x]$ as an $r^{\text{th}}$ order cohomology operation applied to $x$.

In Section~\ref{se:sparse}, we show that when certain sparseness assumptions are made, the subset\break $\lan \Si^{r-1} d_1 \fix \ldots \fix \Si d_1 \fix d_1 , x \ran$ coincides with the full Toda bracket, and give examples of this phenomenon. See Theorem~\ref{th:SparseDrProj}, Proposition~\ref{pr:HomotCoproduct}, and Example~\ref{ex:Compact}. The main application is to computing maps in the homotopy category of $R$-module spectra, for a ring spectrum $R$ whose coefficient ring $\pi_* R$ is sufficiently sparse, such as $ku$. See Example~\ref{ex:RingSpectrum}.

In Appendix~\ref{se:StableMod}, we compute examples of Toda brackets in stable module categories. In particular, Proposition~\ref{pr:proper-inclusion-C4} provides an example where the inclusion $d_2 [x] \subseteq \lan \Si d_1, d_1, x \ran$ is proper. Appendix~\ref{se:Heller} provides for the record a short, simple proof of a theorem due to Heller, that $3$-fold Toda brackets determine the triangulated structure. As a corollary, we note that the $3$-fold Toda brackets indirectly determine the higher Toda brackets.

\subsection*{Related work}

Detailed treatments of secondary operations can be found in \cite{Adams60}*{\S 3.6},
where Adams used secondary cohomology operations to solve the Hopf invariant one problem,
\cite{MosherT68}*{Chapter 16}, and \cite{Harper02}*{Chapter~4}.

There are various approaches to higher order cohomology operations and higher Toda brackets in the literature, many of which use some form of enrichment in spaces, chain complexes, or groupoids; see for instance \cite{Spanier63}, \cite{Maunder63}, \cite{Kochman80}, and \cite{Klaus01}. In this paper, we work solely with the triangulated structure, without enhancement, and provide no comparison to those other approaches.

In \cite{BauesJ06} and \cite{BauesJ11}, Baues and Jibladze express the Adams $d_2$ in terms of
secondary cohomology operations, and this is generalized to higher differentials by Baues and Blanc in~\cite{BauesB15}.
Their approach starts with an injective resolution as in Diagram~\eqref{eq:InjResol},
and witnesses the equations $d_1 d_1 = 0$ by providing suitably coherent null-homotopies,
described using mapping spaces.
Using this coherence data, the authors express a representative of $d_r [x]$ as a
specific element of the Toda bracket $\lan \Si^{r-1} d_1, \ldots, \Si d_1, d_1, x \ran$.
While this approach makes use of an enrichment, 
we suspect that by translating the (higher dimensional) null-homotopies into lifts to fibers or extensions to cofibers,
one could relate their expression for $d_r [x]$ to ours.

\subsection*{Acknowledgments}

We thank Robert Bruner, Dan Isaksen, 
Peter Jorgensen, 
Fernando Muro, 
Irakli Patchkoria, 
Steffen Sagave, and Dylan Wilson for helpful conversations, as well as the referee for their useful comments.  The second author also thanks the Max-Planck-Institut f\"ur Mathematik for its hospitality.
The second author was partially funded by a grant of the DFG SPP 1786: Homotopy Theory and Algebraic Geometry.

\section{The Adams spectral sequence}\label{se:AdamsSS}

In this section, we recall the construction of the Adams spectral sequence in a triangulated category, along with some of its features. We follow~\cite{Christensen98}*{\S 4}, or rather its dual. Some references for the classical Adams spectral sequence are~\cite{Adams74}*{\S III.15}, \cite{Margolis83}*{Chapter 16}, and~\cite{Bruner09}. %
Background material on triangulated categories can be found in \cite{Neeman01}*{Chapter 1}, 
\cite{Margolis83}*{Appendix 2}, and \cite{Weibel94}*{Chapter 10}.
We assume that the suspension functor $\Si$ is an equivalence, with chosen inverse $\Si^{-1}$.
Moreover, we assume we have chosen natural isomorphisms $\Si \Si^{-1} \cong \id$ and $\Si^{-1} \Si \cong \id$
making $\Si$ and $\Si^{-1}$ into an adjoint equivalence.
We silently use these isomorphisms when needed, e.g., when we say that a triangle
of the form $\Si^{-1} Z \to X \to Y \to Z$ is distinguished.

\begin{defn}[{\cite{Christensen98}*{Proposition 2.6}}]
A \Def{projective class} in a triangulated category $\cat{T}$ is a pair $(\cat{P}, \cat{N})$ where $\cat{P}$ is a class of objects and $\cat{N}$ is a class of maps satisfying the following properties.
\begin{enumerate}
\item A map $f \colon X \to Y$ is in $\cat{N}$ if and only if the induced map
\[
f_* \colon \cat{T}(P,X) \to \cat{T}(P,Y)
\]
is zero for all $P$ in $\cat{P}$. In other words, $\cat{N}$ consists of the \Def{$\cat{P}$-null} maps.
\item An object $P$ is in $\cat{P}$ if and only if the induced map
\[
f_* \colon \cat{T}(P,X) \to \cat{T}(P,Y)
\]
is zero for all $f$ in $\cat{N}$.
\item For every object $X$, there is a distinguished triangle $P \to X \ral{f} Y \to \Si P$, where $P$ is in $\cat{P}$ and $f$ is in $\cat{N}$.
\end{enumerate}
\end{defn}

In particular, the class $\cat{P}$ is closed under arbitrary coproducts and retracts. The objects in $\cat{P}$ are called \Def{projective}.

\begin{defn}
A projective class is \Def{stable} if it is closed under shifts, i.e., $P \in \cat{P}$ implies $\Si^n P \in \cat{P}$ for all $n \in \Z$.%
\end{defn}

We will assume for convenience that our projective class is stable. We suspect that many of the results can be adapted to unstable projective classes, with a careful treatment of shifts. 

\begin{defn}
Let $\cat{P}$ be a projective class. A map $f \colon X \to Y$ is called
\begin{enumerate}
\item \Def{$\cat{P}$-epic} if the map
\[
f_* \colon \cat{T}(P,X) \to \cat{T}(P,Y)
\]
is surjective for all $P \in \cat{P}$. Equivalently, the map to the cofiber $Y \to C_f$ is $\cat{P}$-null.
\item \Def{$\cat{P}$-monic} if the map
\[
f_* \colon \cat{T}(P,X) \to \cat{T}(P,Y)
\]
is injective for all $P \in \cat{P}$. Equivalently, the map from the fiber $\Si^{-1} C_f \to X$ is $\cat{P}$-null.
\end{enumerate}
\end{defn}

\begin{ex}\label{ex:GhostSphere}
Let $\cat{T}$ be the stable homotopy category and $\cat{P}$ the projective class consisting of retracts of wedges of spheres $\bigvee_i S^{n_i}$. This is called the \emph{ghost projective class}, studied for instance in \cite{Christensen98}*{\S 7}.%
\end{ex}

Now we dualize everything.

\begin{defn}
An \Def{injective class} in a triangulated category $\cat{T}$ is a projective class in the opposite category $\cat{T}^{\opp}$. Explicitly, it is a pair $(\cat{I}, \cat{N})$ where $\cat{I}$ is a class of objects and $\cat{N}$ is a class of maps satisfying the following properties.
\begin{enumerate}
\item A map $f \colon X \to Y$ is in $\cat{N}$ if and only if the induced map
\[
f^* \colon \cat{T}(Y,I) \to \cat{T}(X,I)
\]
is zero for all $I$ in $\cat{I}$.
\item An object $I$ is in $\cat{I}$ if and only if the induced map
\[
f^* \colon \cat{T}(Y,I) \to \cat{T}(X,I)
\]
is zero for all $f$ in $\cat{N}$.
\item For every object $X$, there is a distinguished triangle $\Si^{-1} I \to W \ral{f} X \to I$, where $I$ is in $\cat{I}$ and $f$ is in $\cat{N}$.
\end{enumerate}
\end{defn}

In particular, the class $\cat{I}$ is closed under arbitrary products and retracts. The objects in $\cat{I}$ are called \Def{injective}.
Just as for projective classes, we will assume for convenience that our injective class is stable.

\begin{ex}\label{ex:InjClass}
Let $\cat{T}$ be the stable homotopy category. Take $\cat{N}$ to be the class of maps inducing zero on mod $p$ cohomology and $\cat{I}$ to be the retracts of (arbitrary) products $\prod_i \Si^{n_i} H\F_p$ with $n_i \in \Z$. One can generalize this example to any cohomology theory (spectrum) $E$ instead of $H\F_p$, letting $\cat{I}_E$ denote the injective class consisting of retracts of products $\prod_i \Si^{n_i} E$.
\end{ex}

\begin{defn}
Let $\cat{I}$ be an injective class. A map $f \colon X \to Y$ is called
\begin{enumerate}
\item \Def{$\cat{I}$-monic} if the map
\[
f^* \colon \cat{T}(Y,I) \to \cat{T}(X,I)
\]
is surjective for all $I \in \cat{I}$. Equivalently, the map from the fiber $\Si^{-1} C_f \to X$ is $\cat{I}$-null.
\item \Def{$\cat{I}$-epic} if the map
\[
f^* \colon \cat{T}(Y,I) \to \cat{T}(X,I)
\]
is injective for all $I \in \cat{I}$. Equivalently, the map to the cofiber $Y \to C_f$ is $\cat{I}$-null.
\end{enumerate}
\end{defn}

\begin{rem}
The projectives and $\cat{P}$-epic maps determine each other via the lifting property
\[
\xymatrix{
& X \ar@{->>}[d]^f \\
P \ar@{-->}[ur] \ar@{->}[r] & Y. \\
}
\]
Dually, the injectives and $\cat{I}$-monic maps determine each other via the extension property
\[
\xymatrix{
X \ar[r] \ar@{ >->}[d]_f & I \\
Y . \ar@{-->}[ur] &  \\
}
\]
This is part of the equivalent definition of a projective (resp.\ injective) class described in \cite{Christensen98}*{Proposition 2.4}. 
\end{rem}

\begin{convention}\label{co:SuspensionIso}
We will implicitly use the natural isomorphism $\cat{T}(A,B) \cong \cat{T}(\Si^k A, \Si^k B)$ sending a map $f$ to $\Si^k f$. 
\end{convention}

\begin{defn}\label{def:AdamsResol}
An \Def{Adams resolution} of an object $X$ in $\cat{T}$ with respect to a projective class $(\cat{P}, \cat{N})$ is a diagram
\begin{equation} \label{eq:AdamsResolProj}
\cxymatrix{
\mathllap{X =\ }X_0 \ar[rr]^{i_0} & & X_1 \circar[dl]^{\de_0} \ar[rr]^{i_1} & & X_2 \circar[dl]^{\de_1} \ar[rr]^{i_2} & & X_3 \circar[dl]^{\de_2} \ar[r] & \cdots \\
& P_0 \ar@{->>}[ul]^{p_0} & & P_1 \ar@{->>}[ul]^{p_1} & & P_2 \ar@{->>}[ul]^{p_2} & & \\
}
\end{equation}
where every $P_s$ is projective, every map $i_s$ is in $\cat{N}$, and every triangle $P_s \ral{p_s} X_s \ral{i_s} X_{s+1} \ral{\de_s} \Si P_s$ is distinguished. Here the arrows $\de_s \colon X_{s+1} \circto P_{s}$ denote degree-shifting maps, namely, maps $\de_s \colon X_{s+1} \to \Si P_{s}$.

Dually, an \Def{Adams resolution} of an object $Y$ in $\cat{T}$ with respect to an injective class $(\cat{I}, \cat{N})$ is a diagram
\begin{equation}\label{eq:AdamsResolInj}
\cxymatrix{
\mathllap{Y =\ }Y_0 \ar@{ >->}[dr]_{p_0} & & Y_1 \ar@{ >->}[dr]_{p_1} \ar[ll]_{i_0} & & Y_2 \ar@{ >->}[dr]_{p_2} \ar[ll]_{i_1} & & Y_3 \ar[ll]_{i_2} & \cdots \ar[l] \\
& I_0 \circar[ur]_{\de_0} & & I_1 \circar[ur]_{\de_1} & & I_2 \circar[ur]_{\de_2} & & \\
}
\end{equation}
where every $I_s$ is injective, every map $i_s$ is in $\cat{N}$, and every triangle $\Si^{-1} I_s \ral{\Si^{-1} \de_s} Y_{s+1} \ral{i_s} Y_s \ral{p_s} I_s$ is distinguished.%
\end{defn}
From now on, fix a triangulated category $\cat{T}$ and a (stable) injective class $(\cat{I}, \cat{N})$ in $\cat{T}$. 
By repeatedly using condition (3) in the definition of an injective class,
we get the following lemma.

\begin{lem}
Every object $Y$ of $\cat{T}$ admits an Adams resolution.
\end{lem}

Given an object $X$ and an Adams resolution of $Y$, applying $\cat{T}(X,-)$ yields an exact couple
\[
\xymatrix{
\bigoplus_{s,t} \cat{T}(\Sigma^{t-s} X, Y_s) \ar[rr]^-{i = \op (i_s)_*} & & 
\bigoplus_{s,t} \cat{T}(\Sigma^{t-s} X, Y_s) \ar[dl]^-{p = \op (p_s)_*} \\
& \bigoplus_{s,t} \cat{T}(\Sigma^{t-s} X, I_s) \ar[ul]^-{\de = \op (\de_s)_*}
}
\]
and thus a spectral sequence with $E_1$ term
\[
E_1^{s,t} %
        = \cat{T}\left( \Si^{t-s} X, I_s \right)
        \cong \cat{T}\left( \Si^t X, \Si^s I_s \right)
\]
and differentials
\[
d_r \colon E_r^{s,t} \to E_r^{s+r, t+r-1}
\]
given by $d_r = p \circ i^{-(r-1)} \circ \de$, where $i^{-1}$ means choosing an $i$-preimage.
This is called the \Def{Adams spectral sequence} with respect to the injective class $\cat{I}$
abutting to $\cat{T}(\Sigma^{t-s} X, Y)$.

\begin{lem}
The $E_2$ term is given by
\[
E_2^{s,t} = \Ext_{\cat{I}}^{s,t}(X,Y) := \Ext_{\cat{I}}^{s}(\Si^t X,Y)
\]
where $\Ext_{\cat{I}}^{s}(X,Y)$ denotes the $s^{\text{th}}$ derived functor of $\cat{T}(X,-)$ (relative to the injective class $\cat{I}$) applied to the object $Y$.
\end{lem}

\begin{proof}
The Adams resolution of $Y$ yields an $\cat{I}$-injective resolution of $Y$
\begin{equation}\label{eq:InjResol}
\xymatrix @C=\bigcol {
0 \ar[r] & Y \ar[r]^-{p_0} & I_0 \ar[r]^-{(\Si p_1) \de_0} & \Si I_1 \ar[r]^-{(\Si^2 p_2) (\Si \de_1)} & \Si^2 I_2 \ar[r] & \cdots \\
} \qedhere
\end{equation}
\end{proof}

\begin{rem}\label{re:NotGenerate}
We do not assume that the injective class $\cat{I}$ generates, 
i.e., that every non-zero object $X$ admits a non-zero map $X \to I$ to an injective. 
Hence, we do not expect the Adams spectral sequence to be conditionally convergent in general; c.f.~\cite{Christensen98}*{Proposition~4.4}.
\end{rem}

\begin{ex}\label{ex:EBased}
Let $E$ be a commutative (homotopy) ring spectrum. 
A spectrum is called \Def{$E$-injective} if it is a retract of $E \sm W$ for some $W$ \cite{HoveyS99}*{Definition 2.22}. 
A map of spectra $f \colon X \to Y$ is called \Def{$E$-monic} if the map $E \sm f \colon E \sm X \to E \sm Y$ is a split monomorphism. 
The $E$-injective objects and $E$-monic maps form an injective class in the stable homotopy category. 
The Adams spectral sequence associated to 
this injective class
is the \emph{Adams spectral sequence based on $E$-homology}, as described in \cite{Ravenel04}*{Definition~2.2.4}, also called the \emph{unmodified Adams spectral sequence} in \cite{HoveyS99}*{\S 2.2}. Further assumptions are needed in order to identify the $E_2$ term as $\Ext$ groups in $E_*E$-comodules.
\end{ex}

\begin{defn}
The \Def{$\cat{I}$-cohomology} of an object $X$ is the family of abelian groups\break $H^I(X) := \cat{T}(X,I)$ indexed by the injective objects $I \in \cat{I}$.

A \Def{primary operation} in $\cat{I}$-cohomology is a natural transformation $H^I(X) \to H^J(X)$ of functors $\cat{T}^{\opp} \to \Ab$. Equivalently, by the (additive) Yoneda lemma, a primary operation is a map $I \to J$ in $\cat{T}$.
\end{defn}

\begin{ex}
The differential $d_1$ is given by primary operations. More precisely, let $x \in E_1^{s,t}$ be a map $x \colon \Si^{t-s} X \to I_s$. Then $d_1(x) \in E_1^{s+1,t}$ is the composite
\[
\xymatrix{
\Si^{t-s} X \ar[r]^-{x} & I_s \ar[r]^-{\de_s} & \Si Y_{s+1} \ar[r]^-{\Si p_{s+1}} & \Si I_{s+1}. \\
}
\]
In other words, $d_1(x)$ is obtained by applying the primary operation $d_1 := (\Si p_{s+1}) \de_s \colon I_s \to \Si I_{s+1}$ to $x$.
\end{ex}

\begin{prop}
A primary operation $\te \colon I \to J$ appears as $d_1 \colon I_s \circto I_{s+1}$ in some Adams resolution if and only if $\te$ admits an $\cat{I}$-epi -- $\cat{I}$-mono factorization.
\end{prop}

\begin{proof}
The condition is necessary by construction. In the factorization $d_1 = (\Si p_{s+1}) \de_s$, the map $\de_s$ is $\cat{I}$-epic while $p_{s+1}$ is $\cat{I}$-monic.

To prove sufficiency, assume given a factorization $\te = iq \colon I \to W \to J$, where $q \colon I \surj W$ is $\cat{I}$-epic and $i \colon W \inj J$ is $\cat{I}$-monic. Taking the fiber of $q$ twice yields the distinguished triangle
\[
\xymatrix{
\Si^{-1} W \ar[r] & Y_0 \ar@{ >->}[r] & I \ar@{->>}[r]^q & W \\
}
\]
which we relabel
\[
\xymatrix{
Y_1 \ar[r]^-{i_0} & Y_0 \ar@{ >->}[r]^-{p_0} & I \ar@{->>}[r]^-{\de_0} & \Si Y_1. \\
}
\]
Relabeling the given map $i \colon W \inj J$ as $\Si p_1 \colon \Si Y_1 \inj \Si I_1$, we can continue the usual construction of an Adams resolution of $Y_0$ as illustrated in Diagram~\eqref{eq:AdamsResolInj}, in which $\te = iq$ appears as the composite $(\Si p_1) \de_0$. Note that by the same argument, for any $s \geq 0$, $\te$ appears as $d_1 \colon I_s \circto I_{s+1}$ in some (other) Adams resolution.
\end{proof}

\begin{ex}
Not every primary operation appears as $d_1$ in an Adams resolution. For example, consider the stable homotopy category with the projective class $\cat{P}$ generated by the sphere spectrum $S = S^0$, that is, $\cat{P}$ consists of retracts of wedges of spheres. The $\cat{P}$-epis (resp. $\cat{P}$-monos) consist of the maps which are surjective (resp. injective) on homotopy groups. The primary  operation $2 \colon S \to S$ does \emph{not} admit a $\cat{P}$-epi -- $\cat{P}$-mono factorization.

Indeed, assume that $2 = iq \colon S \surj W \inj S$ is such a factorization. We will show that this implies $\pi_2 (S/2) = \Z/2 \op \Z/2$, contradicting the known fact $\pi_2 (S/2) = \Z/4$. Here $S/2$ denotes the mod $2$ Moore spectrum, sitting in the cofiber sequence $S \ral{2} S \to S/2$.

By the octahedral axiom applied to the factorization $2 = iq$, there is a diagram
\[
\xymatrix{
S \ar@{=}[d] \ar@{->>}[r]^-{q} & W \ar@{ >->}[d]^{i} \ar[r] & C_q \ar[d]^{\al} \ar@{ >->}[r]^{\de'} & S^1 \ar@{=}[d] \\
S \ar[r]^2 & S \ar@{->>}[d]^{j} \ar[r] & S/2 \ar[d]^{\be} \ar[r]^{\de} & S^1 \\
& C_i \ar@{=}[r] & C_i & \\
}
\]
with distinguished rows and columns. The long exact sequence in homotopy yields $\pi_n C_q = \mbox{}_2 \pi_{n-1} S$, 
where the induced map $\pi_n(\de') \colon \pi_n C_q \to \pi_n S^1$ corresponds to the inclusion $\mbox{}_2 \pi_{n-1} S \inj \pi_{n-1} S$. Likewise, we have $\pi_n C_i = \left( \pi_{n} S \right) / 2$, 
where the induced map $\pi_n(j) \colon \pi_n S \to \pi_n C_i$ corresponds to the quotient map $\pi_{n} S \surj \left( \pi_{n} S \right) / 2$. The defining cofiber sequence $S \ral{2} S \to S/2$ yields the exact sequence
\[
\xymatrix{
\pi_n S \ar[r]^2 & \pi_n S \ar[r] & \pi_n (S/2) \ar[r]^{\pi_n \de} & \pi_{n-1} S \ar[r]^2 & \pi_{n-1} S \\
}
\] 
which in turn yields the short exact sequence
\begin{equation*}%
\xymatrix{
0 \ar[r] & \left( \pi_n S \right) / 2 \ar[r] & \pi_n (S/2) \ar[r]^{\pi_n \de} & \mbox{}_2 \pi_{n-1} S \ar[r] & 0. \\
}
\end{equation*}
The map $\pi_n (\al) \colon \mbox{}_2 \pi_{n-1} S \to \pi_n(S/2)$ is a splitting of this sequence, because of the equality $\pi_n(\de) \pi_n (\al) = \pi_n(\de \al) = \pi_n(\de')$. However, the short exact sequence does not split in the case $n=2$, by the isomorphism $\pi_2(S/2) = \Z/4$. 
For references, see~\cite{Schwede12}*{Proposition II.6.48}, \cite{Schwede10}*{Proposition 4},
and~\cite{MO100272}.
\end{ex}

\section{\texorpdfstring{$3$}{3}-fold Toda brackets}\label{se:3-fold-Toda-brackets}

In this section, we review different constructions of $3$-fold Toda brackets and some of their properties.

\enlargethispage{3pt}
\begin{defn} \label{def:TodaBracket}
Let $X_0 \ral{f_1} X_1 \ral{f_2} X_2 \ral{f_3} X_3$ be a diagram in a triangulated category $\cat{T}$. We define subsets of $\cat{T}(\Si X_0, X_3)$ as follows.
\begin{itemize}
\item The \Def{iterated cofiber Toda bracket} $\lan f_3, f_2, f_1 \ran_{\cc} \subseteq \cat{T}(\Si X_0, X_3)$ consists of all maps $\psi \colon \Si X_0 \to X_3$ that appear in a commutative diagram
\begin{equation} \label{eq:CofCof}
\cxymatrix{
X_0 \ar@{=}[d] \ar[r]^-{f_1} & X_1 \ar@{=}[d] \ar[r] & C_{f_1} \ar[d]^{\phy} \ar[r] & \Si X_0 \ar[d]^{\psi} \\
X_0 \ar[r]^-{f_1} & X_1 \ar[r]^-{f_2} & X_2 \ar[r]^-{f_3} & X_3 \\
}
\end{equation}
where the top row is distinguished.
\item The \Def{fiber-cofiber Toda bracket} $\lan f_3, f_2, f_1 \ran_{\fc} \subseteq \cat{T}(\Si X_0, X_3)$ consists of all composites $\be \circ \Si \al \colon \Si X_0 \to X_3$, where $\al$ and $\be$ appear in a commutative diagram
\begin{equation} \label{eq:FibCof}
\vcenter{
\xymatrix @C=\bigcol {
X_0 \ar[d]_-{\al} \ar[r]^-{f_1} & X_1 \ar@{=}[d] & & \\
\Si^{-1} C_{f_2} \ar[r] & X_1 \ar[r]^-{f_2} & X_2 \ar@{=}[d] \ar[r] & C_{f_2} \ar[d]^-{\be}  \\
& & X_2  \ar[r]^-{f_3} & X_3 \\
}
}
\end{equation}
where the middle row is distinguished.
\item The \Def{iterated fiber Toda bracket} $\lan f_3, f_2, f_1 \ran_{\ff} \subseteq \cat{T}(\Si X_0, X_3)$ consists of all maps $\Si \de \colon \Si X_0 \to X_3$ where $\de$ appears in a commutative diagram
\begin{equation} \label{eq:FibFib}
\vcenter{
\xymatrix @C=\bigcol {
X_0 \ar[d]_{\de} \ar[r]^-{f_1} & X_1 \ar[d]_{\ga} \ar[r]^-{f_2} & X_2 \ar@{=}[d] \ar[r]^-{f_3} & X_3 \ar@{=}[d] \\
\Si^{-1} X_3 \ar[r] & \Si^{-1} C_{f_3} \ar[r] & X_2 \ar[r]^-{f_3} & X_3 \\
}
}
\end{equation}
where the bottom row is distinguished.
\end{itemize}
\end{defn}

\begin{rem}\label{re:3-fold-negation}
In the literature, there are variations of these definitions, which sometimes
differ by a sign.
With the notion of cofiber sequence implicitly used in~\cite{Toda62},
our definitions agree with Toda's.
The Toda bracket also depends on the choice of triangulation.
Given a triangulation, there is an associated negative triangulation whose
distinguished triangles are those triangles whose negatives are distinguished
in the original triangulation (see~\cite{Balmer02}).
Negating a triangulation negates the $3$-fold Toda brackets.
Dan Isaksen has pointed out to us that in the stable homotopy category
there are $3$-fold Toda brackets which are not equal to their own 
negatives.
For example, Toda showed in~\cite{Toda62}*{Section~VI.v, and Theorems~7.4
and~14.1} that the Toda bracket $\lan 2 \sigma, 8, \nu \ran$
has no indeterminacy and contains an element $\zeta$ of order $8$.
We give another example in Example~\ref{ex:negative}.
\end{rem}

The following proposition can be found in \cite{Sagave08}*{Remark 4.5 and Figure 2} and was kindly pointed out by Fernando Muro. It is also proved in \cite{Meier12}*{\S 4.6}. We provide a different proof more in the spirit of this article. In the case of spaces, it was originally proved by Toda \cite{Toda62}*{Proposition 1.7}. 

\begin{prop} \label{TodaBracketsAgree}
The iterated cofiber, fiber-cofiber, and iterated fiber definitions of Toda brackets coincide. More precisely, for any diagram $X_0 \ral{f_1} X_1 \ral{f_2} X_2 \ral{f_3} X_3$ in $\cat{T}$, the following subsets of $\cat{T}(\Si X_0, X_3)$ are equal:
\[
\lan f_3, f_2, f_1 \ran_{\cc} = \lan f_3, f_2, f_1 \ran_{\fc} = \lan f_3, f_2, f_1 \ran_{\ff}.
\]
\end{prop}

\begin{proof}
We will prove the first equality; the second equality is dual.

($\supseteq$) Let $\be (\Si \al) \in \lan f_3, f_2, f_1 \ran_{\fc}$ be obtained from maps $\al$ and $\be$ as in Diagram \eqref{eq:FibCof}. Now consider the diagram with distinguished rows
\[
\xymatrix{
X_0 \ar[d]^{\al} \ar[r]^-{f_1} & X_1 \ar@{=}[d] \ar[r] & C_{f_1} \ar@{-->}[d]^{\phy} \ar[r] & \Si X_0  \ar[d]^{\Si \al} \\
\Si^{-1} C_{f_2} \ar[r] & X_1 \ar[r]^-{f_2} & X_2 \ar@{=}[d] \ar[r] & C_{f_2} \ar[d]^{\be} \\
& & X_2 \ar[r]^-{f_3} & X_3 \\
}
\]
where there exists a filler $\phy \colon C_{f_1} \to X_2$. The commutativity of the tall rectangle on the right exhibits the membership $\be (\Si \al) \in \lan f_3, f_2, f_1 \ran_{\cc}$.

($\subseteq$) Let $\psi \in \lan f_3, f_2, f_1 \ran_{\cc}$ be as in Diagram \eqref{eq:CofCof}. The octahedral axiom comparing the cofibers of $q_1$, $\phy$, and $\phy \circ q_1 = f_2$ yields a commutative diagram
\[
\xymatrix @C=1.1pc @R=0.88pc {
&& && \Si^{-1} C_{\phy} \ar[dd]_{- \Si^{-1} \io} \ar@{=}[rr] & & \Si^{-1} C_{\phy} \ar[dd]_{- \Si^{-1} \eta} & \\ \\
X_0 \ar[dd]_{\al} \ar[rr]^-{f_1} && X_1 \ar@{=}[dd] \ar[rr]^-{q_1} && C_{f_1} \ar[dd]_{\phy} \ar[rr]^-{\io_1} & & \Si X_0  \ar[dd]_{\Si \al} \ar@/_1pc/[dddl]_(0.35){\psi} \ar[rr]^-{- \Si f_1} && \Si X_1 \ar@{=}[dd] \\ \\
\Si^{-1} C_{f_2} \ar[rr]^(0.52){- \Si^{-1} \io_2} && X_1 \ar[rr]^-{f_2} && X_2 \ar[dd]_{q} \ar[dr]_{f_3} \ar[rr]^(0.4){q_2} & & C_{f_2} \ar[dd]^{\xi} \ar@{-->}[dl]^{\!\be} \ar[rr]^{\io_2} && \Si X_1 \\
&& && & X_3 & & \\
&& && C_{\phy} \ar@{-->}[ur]^-{\te} \ar@{=}[rr] & & C_{\phy}, & \\
}
\]
where the rows and columns are distinguished. By exactness of the sequence
\[
\xymatrix @C=\bigcol {
\cat{T}(C_{f_2}, X_3) \ar[r]^-{(\Si \al)^*} & \cat{T}(\Si X_0, X_3) \ar[r]^-{(- \Si^{-1} \eta)^*} & \cat{T}(\Si^{-1} C_{\phy}, X_3)
}
\]
there exists a map $\be \colon C_{f_2} \to X_3$ satisfying $\psi = \be (\Si \al)$ if and only if the restriction of $\psi$ to the fiber $\Si^{-1} C_{\phy}$ of $\Si \al$ is zero. That condition does hold: one readily checks the equality $\psi (- \Si^{-1} \eta) = 0$.
The chosen map $\be \colon C_{f_2} \to X_3$ might \emph{not} satisfy the equation $\be q_2 = f_3$, but we will correct it to another map $\be'$ which does. The error term $f_3 - \be q_2$ is killed by restriction along $\phy$, %
and therefore factors through the cofiber of $\phy$, i.e., there exists a factorization
\[
f_3 - \be q_2 = \te \io
\]
for some $\te \colon C_{\phy} \to X_3$. The corrected map $\be' := \be + \te \xi \colon C_{f_2} \to X_3$ satisfies $\be' q_2 = f_3$. 
Moreover, this corrected map $\be'$ still satisfies $\be' (\Si \al) = \psi = \be (\Si \al)$, since the correction term satisfies $\te \xi (\Si \al) = 0$. %
\end{proof}

Thanks to the proposition, we can write $\lan f_3, f_2, f_1 \ran$ if we 
do not need to specify a particular definition of the Toda bracket.

We also recall this well-known fact, and leave the proof as an exercise:

\begin{lem}\label{le:indeterminacy}
For any diagram $X_0 \ral{f_1} X_1 \ral{f_2} X_2 \ral{f_3} X_3$ in $\cat{T}$,
the subset $\lan f_3, f_2, f_1 \ran$ of $\cat{T}(\Si X_0, X_3)$ is a coset of
the subgroup
\[
  (f_3)_* \, \cat{T}(\Si X_0, X_2) + (\Si f_1)^* \, \cat{T}(\Si X_1, X_3) .
\vspace{-18pt}
\]
\qed
\end{lem}

The displayed subgroup is called the \Def{indeterminacy}, and when
it is trivial, we say that the Toda bracket \Def{has no indeterminacy}.

\begin{lem}\label{le:MoveAround}
Consider maps $X_0 \ral{f_1} X_1 \ral{f_2} X_2 \ral{f_3} X_3 \ral{f_4} X_4$. Then the following inclusions of subsets of $\cat{T}(\Si X_0, X_4)$ hold.
\begin{enumerate}
\item
\[
f_4 \lan f_3, f_2, f_1 \ran \subseteq \lan f_4 f_3, f_2, f_1 \ran
\]
\item
\[
\lan f_4, f_3, f_2 \ran f_1 \subseteq \lan f_4, f_3, f_2 f_1 \ran
\]
\item
\[
\lan f_4 f_3, f_2, f_1 \ran \subseteq \lan f_4, f_3 f_2, f_1 \ran
\]
\item
\[
\lan f_4, f_3, f_2 f_1 \ran \subseteq \lan f_4, f_3 f_2, f_1 \ran.
\]
\end{enumerate}
\end{lem}

\begin{proof}
(1)-(2) These inclusions %
are straightforward.

(3)-(4) Using the iterated cofiber definition, the subset $\lan f_4 f_3, f_2, f_1 \ran_{\cc}$ consists of the maps $\psi \colon \Si X_0 \to X_4$ appearing in a commutative diagram
\[
\xymatrix{
X_0 \ar@{=}[d] \ar[r]^-{f_1} & X_1 \ar@{=}[d] \ar[r] & C_{f_1} \ar[d]^{\phy} \ar[rr] & & \Si X_0 \ar[d]^{\psi} \\
X_0 \ar[r]^-{f_1} & X_1 \ar[r]^-{f_2} & X_2 \ar[r]^-{f_3} & X_3 \ar[r]^-{f_4} & X_4 \\
}
\]
where the top row is distinguished. Given such a diagram, the diagram
\[
\xymatrix{
X_0 \ar@{=}[d] \ar[r]^-{f_1} & X_1 \ar@{=}[d] \ar[r] & C_{f_1} \ar[dr]^{f_3 \phy} \ar[rr] & & \Si X_0 \ar[d]^{\psi} \\
X_0 \ar[r]^-{f_1} & X_1 \ar[r]^-{f_2} & X_2 \ar[r]^-{f_3} & X_3 \ar[r]^-{f_4} & X_4 \\
}
\]
exhibits the membership $\psi \in \lan f_4, f_3 f_2, f_1 \ran_{\cc}$. A similar argument can be used to prove the inclusion $\lan f_4, f_3, f_2 f_1 \ran_{\ff} \subseteq \lan f_4, f_3 f_2, f_1 \ran_{\ff}$.
\end{proof}

\begin{ex}
The inclusion $\lan f_4 f_3, f_2, f_1 \ran \subseteq \lan f_4, f_3 f_2, f_1 \ran$ need not be an equality. For example, consider the maps $X \ral{0} Y \ral{1} Y \ral{0} Z \ral{1} Z$. The Toda brackets being compared are
\begin{align*}
\lan 1_Z 0, 1_Y, 0 \ran &= \lan 0, 1_Y, 0 \ran
= \left\{ 0 \right\} \\
\lan 1_Z, 0 1_Y, 0 \ran &= \lan 1_Z, 0, 0 \ran
= \cat{T}(\Si X, Z).
\end{align*}
\end{ex}

\begin{defn}
In the setup of Definition \ref{def:TodaBracket}, the \Def{restricted Toda brackets} are the subsets of the Toda bracket
\begin{align*}
&\lan f_3, \stackrel{\al}{f_2, f_1} \ran_{\fc} \subseteq \lan f_3, f_2, f_1 \ran_{\fc} \\
&\lan \stackrel{\be}{f_3, f_2}, f_1 \ran_{\fc} \subseteq \lan f_3, f_2, f_1 \ran_{\fc}
\end{align*}
consisting of all composites $\be (\Si \al) \colon \Si X_0 \to X_3$, where $\al$ and $\be$ appear in a commutative diagram \eqref{eq:FibCof} where the middle row is distinguished, with the prescribed map $\al \colon X_0 \to \Si^{-1} C_{f_2} $ (resp. $\be \colon C_{f_2} \to X_3$).
\end{defn}

The lift to the fiber $\al \colon X_0 \to \Si^{-1} C_{f_2}$ is a witness of the equality $f_2 f_1 = 0$. Dually, the extension to the cofiber $\be \colon C_{f_2} \to X_3$ is a witness of the equality $f_3 f_2 = 0$.

\begin{rem} \label{ComposeWitness}
Let $X_1 \ral{f_2} X_2 \ral{q_2} C_{f_2} \ral{\io_2} \Si X_1$ be a distinguished triangle. By definition, we have equalities of subsets
\begin{align*}
&\lan f_3, \stackrel{\al}{f_2, f_1} \ran_{\fc} = \lan f_3, \stackrel{1}{f_2, {-}}\!\! \Si^{-1} \io_2 \ran_{\fc} (\Si \al) \\
&\lan \stackrel{\be}{f_3, f_2}, f_1 \ran_{\fc}  = \be \lan \stackrel{1}{q_2, f_2}, f_1 \ran_{\fc}.
\end{align*}
\end{rem}

\section{Adams \texorpdfstring{$d_2$}{d2} in terms of \texorpdfstring{$3$}{3}-fold Toda brackets}\label{se:AdamsD2}

In this section, we show that the Adams differential $d_r$ can be expressed in several ways 
using $3$-fold Toda brackets.  One of these expressions is as a secondary cohomology operation.

Given an injective class $\cat{I}$,
an Adams resolution of an object $Y$ as in Diagram~\eqref{eq:AdamsResolInj}, and an object $X$, consider a class $[x] \in E_2^{s,t}$ represented by a cycle $x \in E_1^{s,t} = \cat{T}(\Si^{t-s} X, I_s)$. Recall that $d_2 [x] \in E_2^{s+2,t+1} $ is obtained as illustrated in the diagram 
\begin{equation*}
\xymatrix{
\cdots & Y_s \ar[l] \ar@{ >->}[dr]_{p_s} & & Y_{s+1} \ar@{ >->}[dr]_{p_{s+1}} \ar[ll]_{i_s} & & Y_{s+2} \ar@{ >->}[dr]_{p_{s+2}} \ar[ll]_{i_{s+1}} & & Y_{s+3} \ar[ll]_{i_{s+2}} & \cdots \ar[l] \\
& & I_s \circar[ur]_{\de_s} & & I_{s+1} \circar[ur]_{\de_{s+1}} & & I_{s+2} \circar[ur]_{\de_{s+2}} & & \\
& & \Si^{t-s} X \ar[u]_x \ar@/^0.5pc/@{-->}[uurrr]_(0.4){\tild{x}} \ar@/_0.5pc/[urrrr]_{d_2(x)} & & & & & & \\
}
\end{equation*}%
Explicitly, since $x$ satisfies $d_1(x) = (\Si p_{s+1}) \de_s x = 0$, we can choose a lift $\tild{x} \colon \Si^{t-s} X \circto \Si Y_{s+2}$ of $\de_s x$ to the fiber of $\Si p_{s+1}$. Then the differential $d_2$ is given by 
\[
d_2 [x] = \left[ (\Si p_{s+2}) \tild{x} \right].
\]

From now on, we will unroll the distinguished triangles and keep track of the suspensions. %
Following Convention \ref{co:SuspensionIso}, we will use the identifications
\[
E_1^{s+2,t+1} = \cat{T}(\Si^{t-s-1} X, I_{s+2}) \cong \cat{T}(\Si^{t-s} X, \Si I_{s+2}) \cong \cat{T}(\Si^{t-s+1} X, \Si^2 I_{s+2}).
\]

\begin{prop} \label{pr:DifferentD2}
Denote by $d_2 [x] \subseteq E_1^{s+2,t+1}$ the subset of all representatives of the class $d_2 [x] \in E_2^{s+2,t+1}$. Then the following equalities hold:
\begin{enumerate}
\item\label{it:d1pdex}
\begin{align*}
d_2 [x] &= \lan \stackrel{\Si^2 p_{s+2}}{\Si d_1,\strut \Si p_{s+1}}, \de_s x \ran_{\fc} \\
&= \lan \Si d_1, \Si p_{s+1}, \de_s x \ran
\end{align*}
\item
\begin{align*}
d_2 [x] &= (\Si^2 p_{s+2}) \lan \stackrel{\!\!1}{\Si \de_{s+1}, \Si p_{s+1}}, \de_s x \ran_{\fc} \\
&= (\Si^2 p_{s+2}) \lan \Si \de_{s+1}, \Si p_{s+1}, \de_s x \ran
\end{align*}
\item\label{it:d1d1x}
\[
d_2 [x] = \lan \stackrel{\ \be}{\Si d_1, d_1}, x \ran_{\fc} ,
\]
where $\be$ is the composite $C \ral{\tild{\be}} \Si^{2} Y_{s+2} \ral{\Si^2 p_{s+2}} \Si^{2} I_{s+2}$ and $\tild{\be}$ is obtained from the octahedral axiom applied to the factorization $d_1 = (\Si p_{s+1}) \de_{s} \colon I_s \to \Si Y_{s+1} \to \Si I_{s+1}$.
\end{enumerate}
\end{prop}

In~\eqref{it:d1d1x}, $\beta$ is a witness to the fact that the composite $(\Si d_1) d_1$ 
of primary operations is zero, and so the restricted Toda bracket is a secondary operation.

\begin{proof}
Note that $t$ plays no role in the statement, so we will assume without loss of generality that $t=s$ holds.
\smallskip

(1) The first equality holds by definition of $d_2 [x]$, namely choosing a lift of $\de_s x$ to the fiber of $\Si p_{s+1}$. 
The second equality follows from the fact that $\Si^2 p_{s+2}$ is the \emph{unique} extension of $\Si d_1 = (\Si^2 p_{s+2}) (\Si \de_{s+1})$ to the cofiber of $\Si p_{s+1}$. Indeed, $\Si \de_{s+1}$ is $\cat{I}$-epic and $\Si I_{s+2}$ is injective, so that the restriction map
\[
(\Si \de_{s+1})^* \colon \cat{T}(\Si^2 Y_{s+2}, \Si^2 I_{s+2}) \to \cat{T}(\Si I_{s+1}, \Si^2 I_{s+2})
\]
is injective.
\smallskip

(2) The first equality holds by Remark \ref{ComposeWitness}. The second equality holds because $\Si \de_{s+1}$ is $\cat{I}$-epic and $\Si I_{s+2}$ is injective, as in part (1).
\smallskip

(3) The map $d_1 \colon I_s \to \Si I_{s+1}$ is the composite $I_s \ral{\de_s} \Si Y_{s+1} \ral{\Si p_{s+1}} \Si I_{s+1}$. The octahedral axiom applied to this factorization yields the dotted arrows in a commutative diagram
\[
\cxymatrix{
I_s \ar@{=}[d] \ar[r]^-{\de_s} & \Si Y_{s+1} \ar[d]_{\Si p_{s+1}} \ar[r]^-{\Si i_s} & \Si Y_s \ar@{-->}[d]^{\tild{\al}} \ar[r]^-{-\Si p_s} & \Si I_s \ar@{=}[d] \\
I_s \ar[r]^-{d_1} & \Si I_{s+1} \ar[d]_{\Si \de_{s+1}} \ar@{-->}[r]^-{q} & C_{d_1} \ar@{-->}[d]^{\tild{\be}} \ar@{-->}[r]^-{\io} & \Si I_s \\
& \Si^2 Y_{s+2} \ar[d]_{-\Si^{2} i_{s+1}} \ar@{=}[r] & \Si^2 Y_{s+2} \ar[d] & \\
& \Si^2 Y_{s+1} \ar[r]^{\Si^{2} i_s} & \Si^2 Y_{s} & \\
}
\]
where the rows and columns are distinguished and the equation $(-\Si^2 i_{s+1}) \tild{\be} = (\Si \de_s) \io$ holds. The restricted bracket $\lan \stackrel{\ \be}{\Si d_1, d_1}, x \ran_{\fc}$ consists of the maps $\Si X \to \Si^2 I_{s+2}$ appearing as downward composites in the commutative diagram
\[
\xymatrix@C-7pt@R-3pt{
& & & & \Si X \ar@{-->}[d]_-{\Si \al} \ar[rr]^{- \Si x} & & \Si I_s \ar@{=}[d] \\
I_s \ar[rr]^-{d_1} & & \Si I_{s+1} \ar@{=}[dd] \ar[rr]^-{q} & & C_{d_1} \ar[dl]_(0.55){\tild{\be}\!\!} \ar[dd]^{\be} \ar[rr]^-{\io} && \Si I_s \\
& & & \Si^2 Y_{s+2} \ar[dr]^-{\!\Si^2 p_{s+2}} & & \\
& & \Si I_{s+1} \ar[rr]_-{\Si d_1} \ar[ur]^-{\Si \de_{s+1}\!} & & \Si^2 I_{s+2} & \\
}
\]

($\supseteq$) Let $\be (\Si \al) \in \lan \stackrel{\be}{d_1, d_1}, x \ran_{\fc}$. By definition of $\be$, we have $\be (\Si \al) = (\Si^2 p_{s+2}) \tild{\be} (\Si \al)$. Then $\tild{\be} (\Si \al) \colon \Si X \to \Si^{2} Y_{s+2}$ is a valid choice of the lift $\tild{x}$ in the definition of $d_2[x]$:
\begin{align*}
(\Si^2 i_{s+1}) \tild{\be} (\Si \al) &= -(\Si \de_s) \io (\Si \al) \\
&= - (\Si \de_s) (-\Si x) \\
&= \Si (\de_s x).
\end{align*}
($\subseteq$) Given a representative $(\Si p_{s+2}) \tild{x} \in d_2 [x]$, let us show that $\Si \tild{x} \colon \Si X \to \Si^2 Y_{s+2}$ factors as $\Si X \ral{\Si \al} C_{d_1} \ral{\tild{\be}} \Si^{2} Y_{s+2}$ for some $\Si \al$, yielding a factorization of the desired form:
\begin{align*}
(\Si^2 p_{s+2}) (\Si \tild{x}) &= (\Si^2 p_{s+2}) \tild{\be} (\Si \al) \\
&= \be (\Si \al).
\end{align*}
By construction, the map $(\Si^2 i_s) (-\Si^2 i_{s+1}) \colon \Si^2 Y_{s+2} \to \Si^2 Y_{s}$ is a cofiber of $\tild{\be}$. The condition
\[
(\Si^2 i_s) (\Si^2 i_{s+1}) (\Si \tild{x}) = (\Si^2 i_s) \Si (\de_s x) = 0
\]
guarantees the existence of some lift $\Si \al \colon \Si X \to C_{d_1}$ of $\Si \tild{x}$. The chosen lift $\Si \al$ might \emph{not} satisfy $\io (\Si \al) = - \Si x$, but we will correct it to a lift $\Si \al'$ which does. The two sides of the equation become equal after applying $-\Si \de_s$, i.e., $(-\Si \de_s) (-\Si x) = (-\Si \de_s) \io (\Si \al)$ holds.
Hence, the error term factors as
\[
-\Si x - \io \Si \al = (-\Si p_s)(\Si \te)
\]
for some $\Si \te \colon \Si X \to \Si Y_s$, since $-\Si p_s$ is a fiber of $-\Si \de_s$. The corrected map $\Si \al' := \Si \al + \tild{\al} (\Si \te) \colon \Si X \to C_{d_1}$ satisfies $\io (\Si \al') = - \Si x$ 
and still satisfies $\tild{\be} (\Si \al') = \tild{\be} (\Si \al) = \Si \tild{x}$, since the correction term $\tild{\al} (\Si \te)$ satisfies $\tild{\be} \tild{\al} (\Si \te) = 0$. 
\end{proof}

\begin{prop}\label{pr:inclusions}
The following inclusions of subsets hold in $E_1^{s+2,t+1}$: %
\[
d_2 [x] \subseteq (\Si^2 p_{s+2}) \lan \Si \de_{s+1}, d_1, x \ran \subseteq \lan \Si d_1, d_1, x \ran .
\]
\end{prop}

\begin{proof}
The first inclusion is 
\[
d_2 [x] = (\Si^2 p_{s+2}) \lan \Si \de_{s+1}, \Si p_{s+1}, \de_s x \ran \subseteq (\Si^2 p_{s+2}) \lan \Si \de_{s+1}, (\Si p_{s+1}) \de_s, x \ran,
\]
whereas the second inclusion is
\[
(\Si^2 p_{s+2}) \lan \Si \de_{s+1}, d_1, x \ran \subseteq \lan (\Si^2 p_{s+2}) (\Si \de_{s+1}), d_1, x \ran,
\] 
both using Lemma~\ref{le:MoveAround}.
\end{proof}

\begin{prop}\label{pr:proper-inclusion}
The inclusion $(\Si^2 p_{s+2}) \lan \Si \de_{s+1}, d_1, x \ran \subseteq \lan \Si d_1, d_1, x \ran$ need \emph{not} be an equality in general.
\end{prop}

It was pointed out to us by Robert Bruner that this can happen in principle.
We give an explicit example in Proposition~\ref{pr:proper-inclusion-C4}.

\section{Higher Toda brackets}\label{se:HigherBrackets}

We saw in Section~\ref{se:3-fold-Toda-brackets} that there are several equivalent ways
to define $3$-fold Toda brackets.
Following the approach given in~\cite{McKeown12}, we show that
the fiber-cofiber definition generalizes nicely to $n$-fold Toda brackets.
There are $(n-2)!$ ways to make this generalization, and we prove
that they are all the same up to a specified sign.
We also show that this Toda bracket is self-dual.

Other sources that discuss higher Toda brackets in triangulated categories
are~\cite{Shipley02}*{Appendix A}, \cite{Gelfand03}*{IV \S 2} and~\cite{Sagave08}*{\S 4},
which all give definitions that follow Cohen's approach for spectra or spaces~\cite{Cohen68}.
We show that our definition agrees with those of~\cite{Shipley02} and~\cite{Sagave08}.
(We believe that it sometimes differs in sign from~\cite{Cohen68}.  We have not compared carefully with~\cite{Gelfand03}.)

\begin{defn}\label{def:TodaFamily}
Let $X_0 \ral{f_1} X_1 \ral{f_2} X_2 \ral{f_3} X_3$ be a diagram
in a triangulated category $\cat{T}$.
We define the \Def{Toda family}
of this sequence to be the collection $\TF(f_3, f_2, f_1)$
consisting of all pairs $(\beta, \Sigma \alpha)$, where $\alpha$ and $\beta$ appear in a commutative diagram
\[
\xymatrix @C=\bigcol {
X_0 \ar[d]_-{\al} \ar[r]^-{f_1} & X_1 \ar@{=}[d] & & \\
\Si^{-1} C_{f_2} \ar[r] & X_1 \ar[r]^-{f_2} & X_2 \ar@{=}[d] \ar[r] & C_{f_2} \ar[d]^-{\be}  \\
& & X_2  \ar[r]^-{f_3} & X_3 \\
}
\]
with distinguished middle row.
Equivalently,
\[
\xymatrix @C=\bigcol {
& & & \Sigma X_0 \ar[d]_-{\Sigma \al} \ar[r]^-{-\Sigma f_1} & \Sigma X_1 \ar@{=}[d] \\
& X_1 \ar[r]^-{f_2} & X_2 \ar@{=}[d] \ar[r] & C_{f_2} \ar[d]^-{\be} \ar[r] & \Sigma X_1 \\
& & X_2  \ar[r]^-{f_3} & X_3 , \\
}
\]
where the middle row is again distinguished.  (The negative of $\Sigma f_1$
appears, since when a triangle is rotated, a sign is introduced.)
Note that the maps in each pair form a composable sequence
$\Sigma X_0 \ral{\Sigma \alpha} C_{f_2} \ral{\beta} X_3$,
with varying intermediate object,
and that the collection of composites $\beta \circ \Sigma \alpha$ is exactly the
Toda bracket $\langle f_3, f_2, f_1 \rangle$, using the fiber-cofiber definition
(see Diagram~\eqref{eq:FibCof}).
(Also note that the Toda family is generally a proper class,
but this is only because the intermediate object can be varied up to isomorphism,
and so we will ignore this.)

More generally, if $S$ is a set of composable triples of maps,
starting at $X_0$ and ending at $X_3$, we define $\TF(S)$ to
be the union of $\TF(f_3, f_2, f_1)$ for each triple
$(f_3, f_2, f_1)$ in $S$.
\end{defn}

\begin{defn}\label{def:HigherToda}
Let
$X_0 \ral{f_1} X_1 \ral{f_2} X_2 \ral{f_3} \cdots \ral{f_n} X_n$
be a diagram in a triangulated category $\cat{T}$.
We define the \Def{Toda bracket} $\langle f_n, \ldots, f_1 \rangle$
inductively as follows.
If $n = 2$, it is the set consisting of just the composite $f_2 f_1$.
If $n > 2$, it is the union of the sets
$\langle \beta, \Sigma \alpha, \Sigma f_{n-3}, \ldots, \Sigma f_1 \rangle$,
where $(\beta, \Sigma \alpha)$ is in $\TF(f_n, f_{n-1}, f_{n-2})$.
\end{defn}

In fact, there are $(n-2)!$ such definitions, depending on a
sequence of choices of which triple of consecutive maps to apply
the Toda family construction to.
In Theorem~\ref{th:n-fold} we will enumerate these choices
and show that they all agree up to sign.

\begin{ex}\label{ex:4FoldBracket}
Let us describe $4$-fold Toda brackets in more detail. We have
\[
\lan f_4, f_3, f_2, f_1 \ran
 = \bigcup_{\be, \al} \lan \be, \Si \al, \Si f_1 \ran
 = \bigcup_{\be, \al} \bigcup_{\be', \al'} \{ \be' \circ \Si \al' \}
\]
with $(\be, \Si \al) \in \TF(f_4, f_3, f_2)$ and $(\be', \Si \al') \in \TF(\be, \Si \al, \Si f_1)$. These maps fit into a commutative diagram
\[
  \xymatrix{
    \Si^2 X_0 \ar[r]^-{\Si \al'} &	C_{\Si \al} \ar[r] \ar[ddr]^(0.3){\be'} &	\Si^2 X_1 &  \text{row = $\mathrlap{-\Si^2 f_1}$} \\
    \Si X_1 \ar[r]^{\Si \al} &	C_{f_3} \ar[r] \ar[dr]_(0.45){\be} \ar[u] &	\Si X_2 &	\text{row = $\mathrlap{-\Si f_2}$} \\
    X_2 \ar[r]^{f_3} &	X_3 \ar[u] \ar[r]_{f_4} &	X_4 \\
  & 0 \ar[u] \\
  }
\]
where the horizontal composites are specified as above, and each ``snake''
\[
\xymatrix{
& \cdot \ar[r] & \cdot \\
\cdot \ar[r] & \cdot \ar[u] & \\
}
\]
is a distinguished triangle.
The middle column is an example of a \emph{$3$-filtered object} as defined below.
\end{ex}

Next, we will show that Definition \ref{def:HigherToda} coincides with the definitions of higher Toda brackets in~\cite{Shipley02}*{Appendix A} and~\cite{Sagave08}*{\S 4}, which we recall here.

\begin{defn}\label{def:NFiltered}
Let $n \geq 1$ and consider a diagram in $\cat{T}$
\[
\xymatrix{
Y_0 \ar[r]^-{\la_1} &	Y_1 \ar[r]^-{\la_2} &	Y_2 \ar[r] &	\cdots \ar[r]^-{\la_{n-1}} &	Y_{n-1} \\
}  
\]
consisting of $n-1$ composable maps. An \Def{$n$-filtered object} $Y$ based on $(\la_{n-1}, \ldots, \la_1)$ consists of a sequence of maps
\[
\xymatrix{
0 = F_0 Y \ar[r]^-{i_{0}} &	F_{1} Y \ar[r]^-{i_{1}} &	\cdots \ar[r]^-{i_{n-1}} &	F_n Y = Y \\ 
}
\]
together with distinguished triangles
\[
\xymatrix{
F_{j} Y \ar[r]^-{i_j} &	F_{j+1} Y \ar[r]^-{q_{j+1}} & \Si^{j} Y_{n-1-j} \ar[r]^-{e_j} &	\Si F_j Y \\	
}
\]
for $0 \leq j \leq n-1$, such that for $1 \leq j \leq n-1$, the composite
\[
\xymatrix{
\Si^j Y_{n-1-j} \ar[r]^-{e_j} &	\Si F_j Y \ar[r]^-{\Si q_j} &	\Si^{j} Y_{n-j} \\
}
\]
is equal to $\Si^j \la_{n-j}$. In particular, the $n$-filtered object $Y$ comes equipped with maps
\[
\si'_Y \colon Y_{n-1} \cong F_1 Y \to Y 
\]
\[
\si_Y \colon Y = F_n Y \to \Si^{n-1} Y_0.
\]
\end{defn}

\begin{defn}\label{def:HigherTodaSS}
Let $X_0 \ral{f_1} X_1 \ral{f_2} X_2 \ral{f_3} \cdots \ral{f_n} X_n$
be a diagram in a triangulated category $\cat{T}$.
The \Def{Toda bracket} in the sense of Shipley--Sagave $\langle f_n, \ldots, f_1 \rangle_{\Ship} \subseteq \cat{T}(\Si^{n-2} X_0, X_n)$ is the set of all composites appearing in the middle row of a commutative diagram
\[
\xymatrix{
& X_{n-1} \ar[d]_{\si'_X} \ar[dr]^-{f_n} & \\
\Si^{n-2} X_0 \ar[dr]_{\Si^{n-2} f_1} \ar@{-->}[r] & X \ar[d]^{\si_X} \ar@{-->}[r] & X_n \\
& \Si^{n-2} X_1 , & \\
}
\]
where $X$ is an $(n-1)$-filtered object based on $(f_{n-1}, \ldots, f_3, f_2)$.
\end{defn}

\begin{ex}
For a $3$-fold Toda bracket $\lan f_3, f_2, f_1 \ran_{\Ship}$, a $2$-filtered object $X$ based on $f_2$ amounts to a cofiber of $-f_2$, more precisely, a distinguished triangle
\[
\xymatrix{
X_2 \ar[r]^-{\si'_X} & X \ar[r]^-{\si_X} & \Si X_1 \ar[r]^-{\Si f_2} & \Si X_2.
}
\]
Using this, one readily checks the equality $\lan f_3, f_2, f_1 \ran_{\Ship} = \lan f_3, f_2, f_1 \ran_{\fc}$, as noted in \cite{Sagave08}*{Definition 4.5}.
\end{ex}

\begin{ex}\label{ex:4FoldBracketSS}
For a $4$-fold Toda bracket $\lan f_4, f_3, f_2, f_1 \ran_{\Ship}$, a $3$-filtered object $X$ based on $(f_3, f_2)$ consists of the data displayed in the diagram
\[
  \xymatrix{
    & F_3 X = X \ar[r]^-{q_3 = \si_X} & \Si^2 X_1 & \\
    \Si X_1 \ar[r]^-{- \Si^{-1} e_2} & F_2 X \ar[r]^-{q_2} \ar[u]_{i_2} & \Si X_2 & \text{row = $\mathrlap{-\Si f_2}$} \\
    X_2 \ar[r]^-{- \Si^{-1} e_1} & F_1 X \ar[u]_{i_1} \ar[r]^-{q_1}_-{\cong} & X_3 & \text{row = $\mathrlap{-f_3}$} \\
  & F_0 X = 0 , \ar[u]_{i_0} \\
  }
\]
where the two snakes are distinguished. The bracket consists of the maps $\Si^2 X_0 \to X_4$ appearing as composites of the dotted arrows in a commutative diagram
\[
  \xymatrix@R+5pt{
    \Si^2 X_0 \ar@{-->}[r] & X \ar[r]^-{\si_X} \ar@/^1pc/@{-->}[ddr] & \Si^2 X_1 & \text{row = $\mathrlap{\Si^2 f_1}$} \\
    \Si X_1 \ar[r]^-{- \Si^{-1} e_2} & F_2 X \ar[r]^-{q_2} \ar[u] & \Si X_2 & \text{row = $\mathrlap{-\Si f_2}$} \\
    X_2 \ar[r]^-{- f_3} & X_3 \ar[u] \ar[r]^-{f_4} & X_4 & \\
  & 0 , \ar[u] \\
  }
\]
where the two snakes are distinguished. By negating the first and third map in each snake,
this recovers the description in Example \ref{ex:4FoldBracket}, thus proving the equality of subsets
\[
\lan f_4, f_3, f_2, f_1 \ran_{\Ship} = \lan f_4, f_3, f_2, f_1 \ran.
\]
\end{ex}

\begin{prop}
Definitions \ref{def:HigherToda} and \ref{def:HigherTodaSS} agree. In other words, we have the equality
\[
\lan f_n, \ldots, f_1 \ran_{\Ship} = \lan f_n, \ldots, f_1 \ran
\]
of subsets of $\cat{T}(\Si^{n-2} X_0, X_n)$.%
\end{prop}

\begin{proof}
This is a straightforward generalization of Example \ref{ex:4FoldBracketSS}.
\end{proof}

We define the \Def{negative} of a Toda family $\TF(f_3, f_2, f_1)$
to consist of pairs $(\be, -\Si \al)$ for $(\be, \Si \al) \in \TF(f_3, f_2, f_1)$.
(Since changing the sign of two maps in a triangle doesn't affect
whether it is distinguished, it would be equivalent to put the
minus sign with the $\beta$.)

\begin{lem}\label{le:four-fold}
Let
$X_0 \ral{f_1} X_1 \ral{f_2} X_2 \ral{f_3} X_3 \ral{f_4} X_4$
be a diagram in a triangulated category $\cat{T}$.
Then the two sets of pairs
$\TF(\TF(f_4, f_3, f_2), \Sigma f_1)$ and
$\TF(f_4, \TF(f_3, f_2, f_1))$ are negatives of each other.
\end{lem}

This is stronger than saying that the two ways of computing the Toda bracket
$\langle f_4, f_3, f_2, f_1 \rangle$ are negatives, and the stronger statement
will be used inductively to prove Theorem~\ref{th:n-fold}.

\begin{proof}
We will show that the negative of
$\TF(\TF(f_4, f_3, f_2), \Sigma f_1)$
is contained in the family
$\TF(f_4, \TF(f_3, f_2, f_1))$.
The reverse inclusion is proved dually.

Suppose $(\beta, \Sigma \alpha)$ is in $\TF(\TF(f_4, f_3, f_2), \Sigma f_1)$,
that is, $(\beta, \Sigma \alpha)$ is in $\TF(\beta', \Sigma \alpha', \Sigma f_1)$
for some $(\beta', \Sigma \alpha')$ in $\TF(f_4, f_3, f_2)$.
This means that we have the following commutative diagram
\[
  \xymatrix@!@C=-0.2ex@R=-0.2ex{
                 &&                           && \Sigma X_1 \ar[rr]^{-\Sigma f_2} \ar@{-->}[dd]^{\Sigma \alpha'} && \Sigma X_2 \ar@{=}[dd] \\
                                                                                                                           \\
X_2 \ar[rr]^{f_3} && X_3 \ar[dr]_-{f_4} \ar[rr] && C_{f_3} \ar@{-->}[dl]^{\!\beta'} \ar[rr] \ar[dd]        && \Sigma X_2 \\
                 &&             & X_4                                                                                      \\
                 &&                         && C_{\Sigma \alpha'} \ar@{-->}[ul]_{\!\beta} \ar[dd] \\
                 &&             & \Sigma^2 X_0 \ar@{-->}[ur]^{\Sigma \alpha} \ar[dr]_-{- \Sigma^2 f_1}                                  \\
                 &&                         && \Sigma^2 X_1 && \\
  }
\]
in which the long row and column are distinguished triangles.

Using the octahedral axiom, there exists a map $\delta : C_{f_2} \to X_3$
in the following diagram making the two squares commute
and such that the diagram can be extended as shown,
with all rows and columns distinguished:
\[
  \xymatrix@!@C=-0.7ex@R=-0.7ex{
                        &&        & \Sigma X_0 \ar@{-->}[dl]_{\gamma\!} \ar[dr]^{\!-\Sigma f_1}\\
X_2 \ar[rr] \ar@{=}[dd] && C_{f_2} \ar[rr] \ar@{-->}[dd]^{\delta} && \Sigma X_1 \ar[rr]^{-\Sigma f_2} \ar@{-->}[dd]^{\Sigma \alpha'} && \Sigma X_2 \ar@{=}[dd] \\
                                                                                                                           \\
X_2 \ar[rr]^{f_3}  && X_3 \ar[dr]^-{f_4} \ar[rr] \ar[dd] && C_{f_3} \ar@{-->}[dl]^{\!\!\beta'} \ar[rr] \ar[dd]        && \Sigma X_2          \\
                 &&             & X_4                                                                                       \\
                 && C_{\delta} \ar@{=}[rr] \ar[dd] && C_{\Sigma \alpha'} \ar@{-->}[ul]_{\!\!\beta} \ar[dd]     &&       \\
                 &&             & \Sigma^2 X_0 \ar@{-->}[ur]^(0.4){\Sigma \alpha\!} \ar[dr]_(0.4){-\Sigma^2 \!f_1\!\!\!\!} \ar@{-->}[dl]_{\Sigma \gamma\!\!\!}     \\
                 && \Sigma C_{f_2} \ar[rr]   && \Sigma^2 X_1   .               &&        \\
  }
\]
Define $\Sigma \gamma$ to be the composite 
$\Sigma^2 X_0 \to C_{\Sigma \alpha'} = C_{\delta} \to \Sigma C_{f_2}$, where the first map is $\Sigma \alpha$.
Then the small triangles at the top and bottom of the last diagram commute as well.
Therefore, $(\delta, \gamma)$ is in $\TF(f_3, f_2, f_1)$.
Moreover, this diagram shows that
$(\beta, - \Sigma \alpha)$ is in $\TF(f_4, \delta, \gamma)$,
completing the argument.
\end{proof}

To concisely describe different ways of computing higher Toda
brackets, we introduce the following notation.
For $0 \leq j \leq n-3$, write $\TF_j(f_n, f_{n-1}, \ldots, f_1)$ for the set of tuples
\[
  \{ (f_n, f_{n-1}, \ldots, f_{n-j+1}, \beta, \Si \alpha, \Si f_{n-j-3}, \ldots, \Si f_1) \},
\]
where $(\beta, \Si \alpha)$ is in $\TF(f_{n-j}, f_{n-j-1}, f_{n-j-2})$.
(There are $j$ maps to the left of the three\break used for the Toda family.)
If $S$ is a set of $n$-tuples of composable maps, we define\break
$\TF_j(S)$ to be the union of the sets $\TF_j(f_n, f_{n-1}, \ldots, f_1)$
for $(f_n, f_{n-1}, \ldots, f_1)$ in $S$.
With this\break notation, the standard Toda bracket $\lan f_n, \ldots, f_1 \ran$ consists of the composites of all the pairs\break occurring in the iterated Toda family
\[
\TF(f_n, \ldots, f_1) := \TF_0(\TF_0(\TF_0(\cdots \TF_0(f_n, \ldots, f_1) \cdots ))).
\]
A general Toda bracket can be written in the form 
$\TF_{j_1}(\TF_{j_2}(\TF_{j_3}(\cdots \TF_{j_{n-2}}(f_n, \ldots, f_1) \cdots )))$,
where $j_1, j_2, \ldots, j_{n-2}$ is a sequence of natural numbers
with $0 \leq j_i < i$ for each $i$.
There are $(n-2)!$ such sequences.

\begin{rem}
There are six ways to compute the five-fold Toda bracket 
$\langle f_5, f_4, f_3, f_2, f_1 \rangle$, as the set of composites
of the pairs of maps in one of the following sets:
\begin{align*}
&\TF_0(\TF_0(\TF_0(f_5, f_4, f_3, f_2, f_1))) =  \TF(\TF(\TF(f_5, f_4, f_3), \Si f_2), \Si^2 f_1),\\
&\TF_0(\TF_0(\TF_1(f_5, f_4, f_3, f_2, f_1))) =  \TF(\TF(f_5, \TF(f_4, f_3, f_2)), \Si^2 f_1),\\
&\TF_0(\TF_1(\TF_1(f_5, f_4, f_3, f_2, f_1))) =  \TF(f_5, \TF(\TF(f_4, f_3, f_2), \Si f_1)),\\
&\TF_0(\TF_1(\TF_2(f_5, f_4, f_3, f_2, f_1))) =  \TF(f_5, \TF(f_4, \TF(f_3, f_2, f_1))),\\
&\TF_0(\TF_0(\TF_2(f_5, f_4, f_3, f_2, f_1))), \quad\text{and}\\
&\TF_0(\TF_1(\TF_0(f_5, f_4, f_3, f_2, f_1))).
\end{align*}
The last two cannot be expressed directly just using $\TF$.
\end{rem}

Now we can prove the main result of this section.

\begin{thm}\label{th:n-fold}
The Toda bracket computed using the sequence $j_1, j_2, \ldots, j_{n-2}$
equals the standard Toda bracket up to the sign $(-1)^{\sum j_i}$.
\end{thm}

\begin{proof}
Let $j_1, j_2, \ldots, j_{n-2}$ be a sequence with $0 \leq j_i < i$ for each $i$.
Lemma~\ref{le:four-fold} tells us that if we replace consecutive entries 
$k, k+1$ with $k, k$ in any such sequence, the two Toda brackets agree up to a sign.
To begin with, we ignore the signs.
We will prove by induction on $\ell$ that the initial portion
$j_1, \ldots, j_\ell$ of such a sequence can be converted into
any other sequence, using just the move allowed by Lemma~\ref{le:four-fold} and its inverse,
and without changing $j_i$ for $i > \ell$.
For $\ell = 1$, there is only one sequence $0$.
For $\ell = 2$, there are two sequences, $0, 0$ and $0, 1$, and Lemma~\ref{le:four-fold} applies.
For $\ell > 2$, suppose our goal is to produce the sequence $j'_1, \ldots, j'_\ell$.
We break the argument into three cases:
\medskip

\noindent
Case 1: $j'_\ell = j_\ell$.  We can directly use the induction hypothesis to
adjust the entries in the first $\ell - 1$ positions.
\medskip

\noindent
Case 2: $j'_\ell > j_\ell$.  By induction, we can change the first $\ell - 1$
entries in the sequence $j$ so that the entry in position $\ell - 1$ is $j_\ell$,
since $j_{\ell} < j'_{\ell} \leq \ell - 1$.
Then, using Lemma~\ref{le:four-fold}, we can change the entry in position $\ell$ to $j_\ell + 1$.
Continuing in this way, we get $j'_\ell$ in position $\ell$, and then we
are in Case 1.
\medskip

\noindent
Case 3: $j'_\ell < j_\ell$.  Since the moves are reversible, this is equivalent to Case 2.

To handle the sign, first note that signs propagate through the Toda family construction.
More precisely, suppose $S$ is a set of $n$-tuples of maps, and let $S'$ be a set obtained
by negating the $k^{\text{th}}$ map in each $n$-tuple for some fixed $k$.
Then $\TF_j(S)$ has the same relationship to $\TF_j(S')$, possibly for a different value of $k$.

As a result, applying the move of Lemma~\ref{le:four-fold} changes the resulting
Toda bracket by a sign.
That move also changes the parity of $\sum_i j_i$.
Since we get a plus sign when each $j_i$ is zero, it follows that the
difference in sign in general is $(-1)^{\sum_i j_i}$.
\end{proof}

An animation of this argument is available at~\cite{Canim}.
It was pointed out by Dylan Wilson that the combinatorial part of the above proof 
is equivalent to the well-known fact that if a binary operation is associative on triples,
then it is associative on $n$-tuples.

\medskip

In order to compare our Toda brackets to the Toda brackets in the opposite
category, we need one lemma.

\begin{lem}\label{le:suspension-of-Toda-family}
Let
$X_0 \ral{f_1} X_1 \ral{f_2} X_2 \ral{f_3} X_3$
be a diagram in a triangulated category $\cat{T}$.
Then the Toda family $\TF(\Si f_3, \Si f_2, \Si f_1)$ is the negative
of the suspension of $\TF(f_3, f_2, f_1)$.
That is, it consists of $(\Si \be, -\Si^{2} \al)$ for $(\be, \Si \al)$ 
in $\TF(f_3, f_2, f_1)$.
\end{lem}

\begin{proof}
Given a distinguished triangle $\Si^{-1} C_{f_2} \ral{k} X_1 \ral{f_2} X_2 \ral{\iota} C_{f_2}$,
a distinguished triangle involving $\Si f_2$ is
\[
C_{f_2} \ral{-\Si k} \Si X_1 \ral{\Si f_2} \Si X_2 \ral{\Si \iota} \Si C_{f_2} .
\]
Because of the minus sign at the left, the maps that arise in the Toda
family based on this triangle are $-\Si^2 \al$ and $\Si \be$,
where $\Si \al$ and $\be$ arise in the Toda family based on the starting triangle.
\end{proof}

Given a triangulated category $\cat{T}$, the opposite category $\cat{T}^{\opp}$
is triangulated in a natural way.  The suspension in $\cat{T}^{\opp}$ is $\Sigma^{-1}$
and a triangle
\[
\xymatrix{Y_0 \ar[r]^{g_1} & Y_1 \ar[r]^{g_2} & Y_2 \ar[r]^-{g_3} & \Sigma^{-1} Y_0}
\]
in $\cat{T}^{\opp}$ is distinguished if and only if the triangle
\[
\xymatrix{\Sigma \Sigma^{-1} Y_0 & Y_1 \ar[l]_-{g_1'} & Y_2 \ar[l]_-{g_2} & \Sigma^{-1} Y_0 \ar[l]_-{g_3} }
\]
in $\cat{T}$ is distinguished, where $g_1'$ is the composite of $g_1$ with
the natural isomorphism $Y_0 \cong \Sigma \Sigma^{-1} Y_0$.

\begin{cor}\label{co:SelfDual}
The Toda bracket is self-dual up to suspension.
More precisely, let
$X_0 \ral{f_1} X_1 \ral{f_2} X_2 \ral{f_3} \cdots \ral{f_n} X_n$
be a diagram in a triangulated category $\cT$.
Then the subset
\[
  \lan f_1, \ldots, f_n \ran^{\cTo} \subseteq \cTo(\Si^{-(n-2)} X_n, X_0)
= \cT(X_0, \Si^{-(n-2)} X_n)
\]
defined by taking the Toda bracket in $\cTo$ is sent to the subset
\[
  \lan f_n, \ldots, f_1 \ran^{\cT} \subseteq \cT(\Si^{n-2} X_0, X_n)
\]
defined by taking the Toda bracket in $\cT$ under the bijection
$\Si^{n-2} : \cT(X_0, \Si^{-(n-2)} X_n) \to \cT(\Si^{n-2} X_0, X_n)$.
\end{cor}

\begin{proof}
First we compare Toda families in $\cT$ and $\cTo$.
It is easy to see that the Toda family
$T^{\cTo}(f_1, f_2, f_3)$
computed in $\cTo$ consists of the pairs
$(\al, \Si^{-1} \be)$ for $(\Si \al, \be)$ in the Toda family
$T^{\cT}(f_3, f_2, f_1)$ computed in $\cT$.
In short, one has to desuspend and transpose the pairs.

Using this, one can see that the iterated Toda family
\[
T^{\cTo}(T^{\cTo} \cdots T^{\cTo}(f_1, f_2, f_3), \ldots, \Si^{-(n-3)} f_n)
\]
is equal to the transpose of
\[
\Si^{-1} T^{\cT}(\Si^{-(n-3)} f_n, \Si^{-1} T^{\cT}(\Si^{-(n-4)} f_{n-1}, \Si^{-1} T^{\cT} \cdots \Si^{-1} T^{\cT}(f_3, f_2, f_1) \cdots ))
\]
By Lemma~\ref{le:suspension-of-Toda-family}, the desuspensions pass
through all of the Toda family constructions, introducing an overall
sign of $(-1)^{1+2+3+\cdots+(n-3)}$, and producing
\[
\Si^{-(n-2)} T^{\cT}(f_n, T^{\cT}(f_{n-1}, T^{\cT} \cdots T^{\cT}(f_3, f_2, f_1) \cdots ))
\]
By Theorem~\ref{th:n-fold}, composing the pairs gives the usual
Toda bracket up to the sign\break $(-1)^{0+1+2+\cdots+(n-3)}$.
The two signs cancel, yielding the result.
\end{proof}

We do not know a direct proof of this corollary.
To summarize, our insight is that
by generalizing the corollary to all $(n-2)!$ methods of computing the
Toda bracket, we were able to reduce the argument to the $4$-fold case (Lemma~\ref{le:four-fold}) and some combinatorics.

\begin{rem}
As with the $3$-fold Toda brackets (see Remark~\ref{re:3-fold-negation}),
the higher Toda brackets depend on the triangulation.
If the triangulation is negated, the $n$-fold Toda brackets change
by the sign $(-1)^n$.
\end{rem}

\section{Higher order operations determine \texorpdfstring{$d_r$}{dr}}\label{se:AdamsDr}

In this section, we show that the higher Adams differentials can be
expressed in terms of higher Toda brackets, in two ways.
One of these expressions is as an $r^{\text{th}}$ order cohomology operation.

Given an injective class $\cat{I}$,
an Adams resolution of an object $Y$ as in Diagram~\eqref{eq:AdamsResolInj}, and an object $X$, consider a class $[x] \in E_r^{s,t}$ represented by an element $x \in E_1^{s,t} = \cat{T}(\Si^{t-s} X, I_s)$.
The class $d_r[x]$ is the set of all $(\Si p_{s+r}) \tild{x}$, where $\tild x$
runs over lifts of $\de_s x$ through the $(r-1)$-fold composite $\Si(i_{s+1} \cdots i_{s+r-1})$
which appears across the top edge of the Adams resolution.

Our first result will be a generalization of
Proposition~\ref{pr:DifferentD2}\eqref{it:d1pdex},
expressing $d_r$ in terms of an $(r+1)$-fold Toda bracket.

\begin{thm}\label{th:d1pdex}
As subsets of $E_1^{s+r,t+r-1}$, we have
\[
d_r [x] = \lan \Si^{r-1} d_1 , \ldots , \Si^2 d_1, \Si d_1 , \Si p_{s+1} ,  \de_s x \ran .
\]
\end{thm}

\begin{proof}
We compute the Toda bracket, applying the Toda family construction starting from
the right, which introduces a sign of $(-1)^{1+2+\cdots+(r-2)}$, by Theorem~\ref{th:n-fold}.
We begin with the Toda family $\TF(\Si d_1, \Si p_{s+1}, \de_s x)$.
There is a distinguished triangle
\[
  \cxymatrix{\Si Y_{s+2} \ar[r]^-{\Si i_{s+1}} & \Si Y_{s+1} \ar[r]^-{\Si p_{s+1}} & \Si I_{s+1} \ar[r]^-{\Si \de_{s+1}} & \Si^2 Y_{s+2},}
\]
with no needed signs.
The map $\Si d_1$ factors through $\Si \de_{s+1}$ as $\Si^2 p_{s+2}$, and this
factorization is unique because $\Si \de_{s+1}$ is $\cat{I}$-epic and $\Si^2 I_{s+2}$ is injective.
The other maps in the Toda family are $\Si x_1$, where $x_1$ is a lift 
of $\de_s x$ through $\Si i_{s+1}$.
So 
\[
  \TF(\Si d_1, \Si p_{s+1}, \de_s x) = \{ (\Si^2 p_{s+2} , \, \Si x_1) \mid x_1 \text{ a lift of $\de_s x$ through $\Si i_{s+1}$}  \}.
\]
(The Toda family also includes $(\Si^2 p_{s+2} \, \phi, \, \phi^{-1} (\Si x_1))$, where $\phi$
is any isomorphism, but these contribute nothing additional to the later computations.)
The composites of such pairs give $d_2[x]$, up to suspension, recovering
Proposition~\ref{pr:DifferentD2}\eqref{it:d1pdex}.

Continuing, for each such pair we compute
\[
\begin{aligned}
  \TF(\Si^2 d_1, \Si^2 p_{s+2}, \Si x_1)
&= -\Si \TF(\Si d_1, \Si p_{s+2}, x_1) \\
&= -\Si \{ (\Si^2 p_{s+3} , \, \Si x_2) \mid x_2 \text{ a lift of $x_1$ through $\Si i_{s+2}$}  \}.
\end{aligned}
\]
The first equality is Lemma~\ref{le:suspension-of-Toda-family}, and the second reuses
the work done in the previous paragraph, with $s$ increased by $1$.
Composing these pairs gives $-d_3[x]$.
The sign which is needed to produce the standard Toda bracket is $(-1)^1$,
and so the signs cancel.

At the next step, we compute
\[
\begin{aligned}
  \TF(\Si^3 d_1, \Si^3 p_{s+3}, -\Si^2 x_2)
&= -\Si^2 \TF(\Si d_1, \Si p_{s+3}, x_2) \\
&= -\Si^2 \{ (\Si^2 p_{s+4} , \, \Si x_3) \mid x_3 \text{ a lift of $x_2$ through $\Si i_{s+3}$}  \}.
\end{aligned}
\]
Again, the composites give $-d_4[x]$.
Since it was a double suspension that passed through the Toda family, no additional
sign was introduced.
Similarly, the sign to convert to the standard Toda bracket is $(-1)^{1+2}$,
and since $2$ is even, no additional sign was introduced.
Therefore, the signs still cancel.

The pattern continues.  
In total, there are $1+2+\cdots+(r-2)$ suspensions that pass through the Toda
family, and the sign to convert to the standard Toda bracket is also based on
that number, so the signs cancel.
\end{proof}

\begin{rem}
Theorem~\ref{th:d1pdex} can also be proved using the definition Toda
brackets based on $r$-filtered objects, 
as in Definitions~\ref{def:NFiltered} and~\ref{def:HigherTodaSS}.
However, one must work in the opposite category $\cTo$.
In that category, there is a unique $r$-filtered object, up to isomorphism,
based on the maps in the Toda bracket.
One of the dashed arrows in the diagram from Definition~\ref{def:HigherTodaSS}
is also unique, and the other corresponds naturally to the choice of lift
in the Adams differential.
\end{rem}

\medskip

In the remainder of this section, we describe the analog of
Proposition~\ref{pr:DifferentD2}\eqref{it:d1d1x}.
We begin by defining restricted higher Toda brackets, in terms of
restricted Toda families.

Consider a Toda family $\TF(g h_1, g_1 h_0, g_0 h)$, where the maps
factor as shown, there are distinguished triangles
\begin{equation}\label{eq:tri}
  \cxymatrix{Z_i \ar[r]^{g_i} & J_i \ar[r]^{h_i} & Z_{i+1} \ar[r]^{k_i} & \Si Z_i}
\end{equation}
for $i = 0, 1$,
and $g$ and $h$ are arbitrary maps $Z_2 \to A$ and $B \to Z_0$, respectively.
This information determines an essentially unique element of the Toda family in the following way.
The octahedral axiom applied to the factorization $g_1 h_0$
yields the dotted arrows in a commutative diagram
\[
\cxymatrix{
J_0 \ar@{=}[d] \ar[r]^-{h_0} & Z_1 \ar[d]_(0.45){g_1} \ar[r]^-{k_0} & \Si Z_0 \ar@{-->}[d]^{\al_2} \ar[r]^-{-\Si g_0} & \Si J_0 \ar@{=}[d] \\
J_0 \ar[r]^-{g_1 h_0} & J_1 \ar[d]_{h_1} \ar@{-->}[r]^-{q} & W_2 \ar@{-->}[d]^{\be_2} \ar@{-->}[r]^-{\io} & \Si J_0 \\
& Z_2 \ar[d]_{k_1} \ar@{=}[r] & Z_2 \ar[d]^{\ga_2} & \\
& \Si Z_1 \ar[r]^{\Si k_0} & \Si^2 Z_0 , & \\
}
\]
where the rows and columns are distinguished and $\ga_2 := (\Si k_0) k_1$.
It is easy to see that $-\Si(g_0 h)$ lifts through $\io$ as $\al_2 (\Si h)$,
and that $g h_1$ extends over $q$ as $g \be_2$.
We define the \Def{restricted Toda family} to be the set 
$\TF(g h_1 \fix g_1 h_0 \fix g_0 h)$ consisting of the pairs $(g \be_2, \, \al_2 (\Si h))$
that arise in this way.
Since $\al_2$ and $\be_2$ come from a distinguished triangle involving a fixed map $\ga_2$,
such pairs are unique up to the usual ambiguity of replacing
the pair with $(g \be_2 \phi, \, \phi^{-1} \al_2 (\Si h))$, where $\phi$ is an isomorphism.
Similarly, given any map $x : B \to J_0$,
we define $\TF(g h_1 \fix g_1 h_0 , x)$ to be the set 
consisting of the pairs  $(g \be_2, \, \Si \al)$,
where $\be_2$ arises as above and $\Si \al$ is any lift of $-\Si x$ through $\io$.

\begin{defn}
Given distinguished triangles as in Equation~\eqref{eq:tri}, for $i = 1, \ldots, n-1$,
and maps $g : Z_n \to A$ and $x : B \to J_1$, we define the \Def{restricted Toda bracket}
\[
\lan g h_{n-1} \fix g_{n-1} h_{n-2} \fix \ldots \fix g_3 h_2 \fix g_2 h_1 , x \ran
\]
inductively as follows:
If $n = 2$, it is the set consisting of just the composite $g h_1 x$.
If $n = 3$, it is the set of composites of the pairs in $\TF(g h_2 \fix g_2 h_1 , x)$.
If $n > 3$, it is the union of the sets
\[
\lan g \be_2 \fix\, \al_2 (\Si h_{n-3}) \fix\, \Si (g_{n-3} h_{n-4}) \fix \ldots , \Si x \ran ,
\]
where $(g \be_2, \al_2 (\Si h_{n-3}))$ is in $\TF(g h_{n-1} \fix g_{n-1} h_{n-2} \fix g_{n-2} h_{n-3})$.
\end{defn}

\begin{rem}
Despite the notation, we want to make it clear that these restricted
Toda families and restricted Toda brackets depend on the choice of
factorizations and on the distinguished triangles in Equation~\eqref{eq:tri}.
Moreover, the elements of the restricted Toda families are not simply pairs,
but also include the factorizations of the maps in those pairs, and
the distinguished triangle involving $\al_2$ and $\be_2$.
This information is used in the $(n-1)$-fold restricted Toda bracket
that is used to define the $n$-fold restricted Toda bracket.
\end{rem}

Recall that the maps $d_1$ are defined to be $(\Si p_{s+1}) \de_s$, and that we
have distinguished triangles
\[
  \cxymatrix{Y_s \ar[r]^{p_s} & I_s \ar[r]^-{\de_s} & \Si Y_{s+1} \ar[r]^{\Si i_s} & \Si Y_s}
\]
for each $s$.
The same holds for suspensions of $d_1$, with the last map
changing sign each time it is suspended.  
Thus for $x : \Si^{t-s} X \to I_s$ in the $E_1$ term, the $(r+1)$-fold restricted Toda bracket
$\lan \Si^{r-1} d_1 \fix \ldots \fix \Si d_1 \fix d_1 , x \ran$
makes sense for each $r$, where we are implicitly using the defining factorizations
and the triangles from the Adams resolution.

\begin{thm}\label{th:AdamsDrCohomOp}
As subsets of $E_1^{s+r,t+r-1}$, we have
\[
d_r [x] = \lan \Si^{r-1} d_1 \fix \ldots \fix \Si d_1 \fix d_1 , x \ran .
\]
\end{thm}

This is a generalization of Proposition~\ref{pr:DifferentD2}\eqref{it:d1d1x}.
The data in the Adams resolution is the witness that the composites of
the primary operations are zero in a sufficiently coherent way to permit
an $r^{\text{th}}$ order cohomology operation to be defined.

\begin{proof}
The restricted Toda bracket $\lan \Si^{r-1} d_1 \fix \ldots \fix \Si d_1 \fix d_1 , x \ran$
is defined recursively, working from the left.
Each of the $r-2$ doubly restricted Toda families has essentially one element.
The first one involves maps $\al_2$, $\be_2$ and $\ga_2$ that form a distinguished
triangle, and $\ga_2$ is equal to $[(-1)^r \Si^r i_{s+r-2}][-(-1)^r \Si^r i_{s+r-1}]$.
We will denote the corresponding maps in the following octahedra $\al_k$,  $\be_k$ and $\ga_k$,
where each $\ga_k$ equals $[(-1)^r \Si^r i_{s+r-k}] \, \ga_{k-1}$,
and so $\ga_k = -(-1)^{rk} \Si^r (i_{s+r-k} \cdots i_{s+r-1})$.
One is left to compute the singly restricted Toda family
$\lan \Si^r p_{s+r} \be_{r-1} \fix\, \al_{r-1} \Si^{r-2} \de_s ,\, \Si^{r-2} x \ran$,
where $\al_{r-1}$ and $\be_{r-1}$ fit into a distinguished triangle
\[
  \cxymatrix{\Si^{r-1 }Y_{s+1} \ar[r]^-{\al_{r-1}} & W_{r-1} \ar[r]^-{\be_{r-1}} & \Si^{r} Y_{s+r} \ar[r]^-{\ga_{r-1}} & \Si^r Y_{s+1} ,}
\]
and $\ga_{r-1} = - \Si^r (i_{s+1} \cdots i_{s+r-1})$.
Thus, to compute the last restricted Toda bracket, one uses the following diagram,
obtained as usual from the octahedral axiom:
\[
\xymatrix@C+24pt{
& & & \Si^{t-s+r-1} X \ar[d]^(0.45){-\Si^{r-1} x} \\
\Si^{r-2} I_{s} \ar@{=}[d] \ar[r]^-{\Si^{r-2} \de_{s}} & \Si^{r-1} Y_{s+1} \ar[d]_{\al_{r-1}} \ar[r]^-{\!(-1)^r \, \Si^{r-1} i_{s}} & \Si^{r-1} Y_{s} \ar@{-->}[d]^{\al_r} \ar[r]^-{- \Si^{r-1} p_{s}} & \Si^{r-1} I_{s} \ar@{=}[d] \\
\Si^{r-2} I_{s} \ar[r]^-{} & W_{r-1} \ar[d]_{\be_{r-1}} \ar@{-->}[r]^-{q_{r-1}} & W_r \ar@{-->}[d]^{\be_r} \ar@{-->}[r]^-{\io_{r-1}} & \Si^{r-1} I_{s} \\
\Si^r I_{s+r} & \Si^r Y_{s+r} \ar[l]_{\Si^r p_{s+r}} \ar[d]_{\ga_{r-1}} \ar@{=}[r] & \Si^r Y_{s+r} \ar[d]^{\ga_r} & \\
& \Si^r Y_{s+1} \ar[r]^{(-1)^r \, \Si^r i_{s}} & \Si^r Y_{s} . & \\
}
\]
Up to suspension, both $d_r[x]$ and the last restricted Toda bracket are computed by
composing certain maps $\tild{x} : \Si^{t-s+r-2} X \to \Si^r Y_{s+r}$ with $\Si^r p_{s+r}$.
For $d_r[x]$, the maps $\tild{x}$ must lift $\Si^{r-1} (\de_s x)$ through $- \ga_{r-1}$.
For the last bracket, the maps $\tild{x}$ are of the form $\be_r y$,
where $y : \Si^{t-s+r-1} X \to W_r$ is a lift of $-\Si^{r-1} x$ through $\io_{r-1}$.
As in the proof of Proposition~\ref{pr:DifferentD2}\eqref{it:d1d1x}, one can
see that the possible choices of $\tild{x}$ coincide.
\end{proof}

We next give a description of $d_r[x]$ using higher Toda brackets defined
using filtered objects, as in Definitions~\ref{def:NFiltered} and~\ref{def:HigherTodaSS}.
The computation of the restricted Toda bracket above produces a sequence 
\begin{equation}\label{eq:W}
  \cxymatrix{0 = W_0 \ar[r]^-{q_0} & W_1 \ar[r]^{q_1} & \cdots  \ar[r]^{q_{r-1}} & W_r , }
\end{equation}
where $W_k$ is the fibre of the $k$-fold composite $\Si^r (i_{s+r-k} \cdots i_{s+r-1})$.
(The map $\ga_k$ may differ in sign from this composite, but that doesn't affect the fibre.)
For each $k$, we have a distinguished triangle
\[
  \xymatrix@C+3pt{W_k \ar[r]^-{q_k} & W_{k+1} \ar[r]^-{\io_k}
 & \Si^{r-1} I_{s+r-k-1} \ar[rrr]^-{-(\Si \al_k)(\Si^{r-1} \de_{s+r-k-1})} &&& \Si W_k ,}
\]
where we extend downwards to $k=0$ by defining $W_1 = \Si^{r-1} I_{s+r-1}$
and using the non-obvious triangle
\[
  \xymatrix@C+7pt{W_0 \ar[r]^-{q_0 = 0} & W_{1} \ar[r]^-{\io_0 = -1}
 & \Si^{r-1} I_{s+r-1} \ar[r]^-{0} & \Si W_0 .}
\]
One can check that 
\[
(\Si \io_{k-1})(-\Si \al_k)(\Si^{r-1} \de_{s+r-k-1}) = (\Si^r p_{s+r-k})(\Si^{r-1} \de_{s+r-k-1})
= \Si^{r-1} d_1 = \Si^k (\Si^{r-k-1} d_1) ,
\]
where $\Si^{r-k-1} d_1$ is the map appearing in the $(k+1)$st spot of the Toda bracket.
In other words, the sequence~\eqref{eq:W} is an $r$-filtered object based on
$(\Si^{r-2} d_1, \ldots, d_1)$.

The natural map $\si_W : W_r \to \Si^{r-1} I_s$ is $\io_{r-1}$,
and the natural map $\si_W' : \Si^{r-1} I_{s+r-1} \cong W_1 \to W_r$ is the composite
$q_{r-1} \cdots q_1 \io_0 = -q_{r-1} \cdots q_1$.
The Toda bracket computed using the filtered object $W$ consists of all
composites appearing in the middle row of this commutative diagram:
\begin{equation}\label{eq:Wlift}
\cxymatrix{
& \Si^{r-1} I_{s+r-1} \ar[d]_{\si'_W} \ar[dr]^-{\Si^{r-1} d_1}  & \\
\Si^{t-s+r-1} X \ar[dr]_{\Si^{r-1} x} \ar@{-->}[r]^-a & W_r \ar[d]^{\si_W} \ar@{-->}[r]_-b & \Si^r I_{s+r} \\
& \Si^{r-1} I_s . & \\
}
\end{equation}
We claim that there is a natural choice of extension $b$.
Since $\Si^{r-1} d_1 = (\Si^r p_{s+r})(\Si^{r-1} \de_{s+r-1})$, it suffices
to extend $\Si^{r-1} \de_{s+r-1}$ over $\si_W'$.
Well, $\be_2$ by definition is an extension of $\Si^{r-1} \de_{s+r-1}$ over $q_1$,
and each subsequent $\be_k$ gives a further extension.
Because $\io_0 = -1$, $-(\Si^r p_{s+r}) \be_r$ is a valid choice for $b$.

On the other hand, as described at the end of the previous proof,
the lifts $a$ of $\Si^{r-1} x$ through $\si_W = \io_{r-1}$, when 
composed with $-(\Si^r p_{s+r}) \be_r$, give exactly
the Toda bracket computed there.

In summary, we have:

\begin{thm}
Given an Adams resolution of $Y$ and $r \geq 2$, there is an associated
$r$-filtered object $W$ and a choice of a map $b$ in Diagram~\eqref{eq:Wlift},
such that for any $X$ and class $[x] \in E_r^{s,t}$,
we have
\[
d_r [x] = \lan \Si^{r-1} d_1 , \ldots , \Si d_1 , d_1 , x \ran ,
\]
where the Toda bracket is computed only using the $r$-filtered object $W$
and the chosen extension $b$.
\end{thm}

\section{Sparse rings of operations}\label{se:sparse}

In this section, we focus on injective and projective classes which 
are generated by an object with a ``sparse'' endomorphism ring.
In this context, we can give conditions under which the restricted Toda bracket
appearing in Theorem~\ref{th:AdamsDrCohomOp} is equal to the unrestricted Toda bracket,
producing a cleaner correspondence between Adams differentials and Toda brackets.
We begin in Subsection~\ref{ss:sparse-injective} by giving the results in the case
of an injective class, and then briefly summarize the dual results in Subsection~\ref{ss:sparse-projective}.
Subsection~\ref{ss:sparse-examples} gives examples.

Let us fix some notation and terminology, also discussed in \cite{Sagave08}, \cite{Patchkoria12}, 
\cite{SchwedeS03}*{\S 2}, 
and \cite{BensonKS04}.

\begin{defn}
Let $N$ be a natural number. A graded abelian group $R_*$ is \Def{$N$-sparse} if $R_*$ is concentrated in degrees which are multiples of $N$, i.e., $R_i = 0$ whenever $i \not\equiv 0 \pmod{N}$.
\end{defn}

\subsection{Injective case}\label{ss:sparse-injective}

\begin{nota}
Let $E$ be an object of the triangulated category $\cat{T}$. Define the \Def{$E$-cohomology} of an object $X$ to be the graded abelian group $E^*X$ given by $E^n X := \cat{T}(X,\Si^n E)$. Postcomposition makes $E^*X$ into a left module over the graded endomorphism ring $E^*E$.%
\end{nota}

\begin{assum}
For the remainder of this subsection, we assume the following.
\begin{enumerate}
\item The triangulated category $\cat{T}$ has infinite %
products.
\item The graded ring $E^* E$ is $N$-sparse for some $N \geq 2$.%
\end{enumerate}
\end{assum}

Let $\cat{I}_E$ denote the injective class generated by $E$, as in Example~\ref{ex:InjClass}. Explicitly, $\cat{I}_E$ consists of retracts of (arbitrary) products $\prod_i \Si^{n_i} E$.

\begin{lem}\label{le:SparseInj}
With this setup, we have:
\begin{enumerate}
\item Let $I$ be an injective object such that $E^* I$ is $N$-sparse. Then $I$ is a retract of a product $\prod_i \Si^{m_i N} E$.
\item If, moreover, $W$ is an object such that $E^*W$ is $N$-sparse, then we have $\cat{T}(W,\Si^t I) = 0$ for $t \not\equiv 0 \pmod{N}$. 
\end{enumerate}
\end{lem}

\begin{proof}
(1) $I$ is a retract of a product $P = \prod_i \Si^{n_i} E$, with a map $\io \colon I \inj P$ and retraction $\pi \colon P \surj I$. Consider the subproduct $P' = \prod_{N \mid n_i} \Si^{n_i} E$, with inclusion $\io' \colon P' \inj P$ (via the zero map into the missing factors) and projection $\pi' \colon P \surj P'$. Then the equality
\[
\io' \pi' \io = \io \colon I \to P
\]
holds, using the fact that $E^*I$ is $N$-sparse. %
Therefore, we obtain $\pi \io' \pi' \io = \pi \io = 1_I$, so that $I$ is a retract of $P'$. %

(2) By the first part, $\cat{T}(W,\Si^t I)$ is a retract of
\begin{align*}
\cat{T}(W,\Si^t \prod_i \Si^{m_i N} E) &= \cat{T}(W,\prod_i \Si^{m_i N+ t} E) \\
&= \prod_i \cat{T}(W,\Si^{m_i N+ t} E) \\ 
&= \prod_i E^{m_i N+ t}W \\
&= 0 ,
\end{align*}
using the assumption that $E^*W$ is $N$-sparse.
\end{proof}

\begin{lem}\label{le:SparseBracket}
Let $I_0 \ral{f_1} I_1 \ral{f_2} I_2 \to \cdots \ral{f_r} I_r$ be a diagram
in $\cat{T}$, with $r \leq N+1$. Assume that each object $I_j$ is injective and 
that each $E^*(I_j)$ is $N$-sparse. Then the iterated Toda family $\TF(f_r, f_{r-1}, \ldots, f_1)$ is either empty or consists of a single composable pair $\Si^{r-2} I_0 \to C \to I_r$, up to automorphism of $C$.
\end{lem}

\begin{proof}
In the case $r=2$, there is nothing to prove, so we may assume $r \geq 3$. The iterated Toda family is obtained by $r-2$ iterations of the $3$-fold Toda family construction. The first iteration computes the Toda family of the diagram
\[
\xymatrix{
I_{r-3} \ar[r]^-{f_{r-2}} & I_{r-2} \ar[r]^-{f_{r-1}}  & I_{r-1} \ar[r]^-{f_{r}} & I_{r}. \\
}
\]
Choose a cofiber of $f_{r-1}$, i.e., a distinguished triangle $I_{r-2} \ral{f_{r-1}} I_{r-1} \to C_1 \to \Si I_{r-2}$.
A lift of $f_{r-2}$ to the fiber $\Si^{-1} C_1$, if it exists, is determined up to
\[
\cat{T}(I_{r-3}, \Si^{-1} I_{r-1}) = \cat{T}(\Si I_{r-3}, I_{r-1}),
\]
which is zero by Lemma~\ref{le:SparseInj}(2).
Likewise, an extension of $f_{r}$ to the cofiber $C_1$, if it exists, is determined up to
\[
\cat{T}(\Si I_{r-2}, I_{r}) = 0.
\]
Hence, $\TF(f_r, f_{r-1}, f_{r-2})$ is either empty or consists of a single pair $(\be_1, \Si \al_1)$,
up to automorphisms of $C_1$.
It is easy to see that the object $C_1$ has the following property:
\begin{equation}\label{eq:SparseCohom1}
\text{If $E^*W = 0$ for $\ast \equiv 0,1 \;(\bmod\; N)$, then $\cat{T}(W,C_1) = 0$.}
\end{equation}
For $r \geq 4$, the next iteration computes the Toda family of the diagram
\[
\xymatrix{
\Si I_{r-4} \ar[r]^-{\Si f_{r-3}} & \Si I_{r-3} \ar[r]^-{\Si \al_1} & C_1 \ar[r]^-{\be_1} & I_{r}. \\
}
\]
The respective indeterminacies are
\[
\cat{T}(\Si^2 I_{r-4}, C_1),
\]
which is zero by Property~\eqref{eq:SparseCohom1}, and 
\[
\cat{T}(\Si^2 I_{r-3}, I_{r}),
\]
which is zero by Lemma~\ref{le:SparseInj}(2), since $N \geq 3$ in this case.
Hence, $\TF(\be_1, \Si \al_1, \Si f_{r-3})$ is either empty or consists of a single pair $(\be_2, \Si \al_2)$,
up to automorphism of the cofiber $C_2$ of $\Sigma \alpha_1$.
Repeating the argument inductively, the successive iterations compute the Toda family of a diagram
\[
\xymatrix @C=\bigcol {
\Si^j I_{r-3-j} \ar[r]^-{\Si^j f_{r-2-j}} & \Si^j I_{r-2-j} \ar[r]^-{\Si \al_j} & C_j \ar[r]^-{\be_j} & I_{r} \\
}
\]
for $0 \leq j \leq r-3$, where $C_j$ has the following property:
\begin{equation}\label{eq:SparseCohomj}
\text{If $E^*W = 0$ for $\ast \equiv 0,1,\ldots,j \;(\bmod\; N)$, then $\cat{T}(W,C_j) = 0$.}
\end{equation}
The indeterminacy groups $\cat{T}(\Si^{j+1} I_{r-3-j}, C_j)$ and $\cat{T}(\Si^{j+1} I_{r-2-j}, I_{r})$
are again zero.
Hence,\break $\TF(\be_j, \Si \al_j, \Si^j f_{r-2-j})$ is either empty or consists of a single pair 
$(\be_{j+1}, \Si \al_{j+1})$, up to automorphism of $C_{j+1}$.
Note that the argument works until the last iteration $j = r-3$, by the assumption $r-2 < N$.
\end{proof}

We will need the following condition on an object $Y$:
\begin{condition}\label{co:InjSparse}
$Y$ admits an $\cat{I}_E$-Adams resolution $Y_{\bu}$ (see \eqref{eq:AdamsResolInj}) such that for each injective $I_j$ in the resolution,
$E^* (\Si^j I_j)$ is $N$-sparse.
\end{condition}

\pagebreak[2]
\begin{rem}\leavevmode
\begin{enumerate}
\item Condition~\ref{co:InjSparse} implies that $E^* Y$ is itself $N$-sparse, because of the surjection $E^* I_0 \surj E^* Y$.
\item The condition can be generalized to: there is an integer $m$ such that for each $j$, $E^* (\Si^j I_j)$ is concentrated in degrees $\ast \equiv m \pmod{N}$. We take $m=0$ for notational convenience.
\item We will see in Propositions~\ref{pr:CohomProduct} and~\ref{pr:CoherentRing} situations in which
Condition~\ref{co:InjSparse} holds.
\end{enumerate}
\end{rem}

\begin{thm}\label{th:SparseDr}
Let $X$ and $Y$ be objects in $\cat{T}$ and consider the Adams spectral sequence abutting to $\cat{T}(X,Y)$ with respect to the injective class $\cat{I}_E$. Assume that $Y$ satisfies Condition~\ref{co:InjSparse}. Then for all $r \leq N$, the Adams differential is given, as a subset of $E_1^{s+r,t+r-1}$, by
\[
d_r [x] = \lan \Si^{r-1} d_1 , \ldots, \Si d_1, d_1 , x \ran.
\]
In other words, the restricted bracket appearing in Theorem~\ref{th:AdamsDrCohomOp} coincides with the full Toda bracket.
\end{thm}

\begin{proof}
We will show that 
\[
  \lan \Si^{r-1} d_1 \fix \ldots \fix \Si d_1 \fix d_1 , x \ran = \lan \Si^{r-1} d_1 , \ldots, \Si d_1, d_1 , x \ran.
\]
Consider the diagram 
\[
\xymatrix{
I_s \ar[r]^-{d_1} & \Si I_{s+1} \ar[r]^-{\Si d_1} & \Si^2 I_{s+2} \ar[r] & \cdots \ar[r] & \Si^{r-1} I_{r-1} \ar[r]^-{\Si^{r-1} d_1} & \Si^{r} I_{s+r} \\
X \ar[u]^{x} & & & & & \\
}
\]
whose Toda bracket is being computed. The corresponding Toda family is
\[
\TF(\Si^{r-1} d_1 , \ldots, \Si d_1, d_1 , x) = \TF \left( \TF(\Si^{r-1} d_1 , \ldots, \Si d_1, d_1), \Si^{r-2} x \right).
\]
We know that
\[
  \TF(\Si^{r-1} d_1 \fix \ldots \fix \Si d_1 \fix d_1) \subseteq \TF(\Si^{r-1} d_1 , \ldots, \Si d_1, d_1).
\]
By Lemma~\ref{le:SparseBracket}, %
the Toda family on the right has at most one element, up to automorphism.
But fully-restricted Toda families are always non-empty, so the inclusion must be an equality.
Write $\Si^{r-2} I_{s} \ral{f} C \ral{g} \Si^r I_{s+r}$ for an element of these families.
It remains to show that the inclusion
\[
  \lan g \fix f, \Si^{r-2} x \ran \subseteq \lan g, f, \Si^{r-2} x \ran 
\]
is an equality, i.e., that the extension of $g$ to the cofiber of $f$ is unique.
This follows from the equality $\cat{T}(\Si^{r-1} I_{s}, \Si^r I_{s+r}) = 0$, which uses the assumption on the injective objects $I_j$ and that $r-1 < N$.
\end{proof}

Next, we describe situations in which Theorem~\ref{th:SparseDr} applies.

\begin{prop}\label{pr:CohomProduct}
Assume that every product of the form $\prod_i \Si^{m_i N} E$ has cohomology\break $E^* \left( \prod_i \Si^{m_i N} E \right)$ which is $N$-sparse. Then every object $Y$ such that $E^* Y$ is $N$-sparse also satisfies Condition~\ref{co:InjSparse}.
\end{prop}

\begin{proof}
Let $(y_i)$ be a set of non-zero generators of $E^* Y$ as an $E^* E$-module. %
Then the corresponding map $Y \to \prod_i \Si^{\abs{y_i}} E$ is $\cat{I}_E$-monic into an injective object;
we take this map as the first step $p_0 \colon Y_0 \to I_0$, with cofiber $\Si Y_1$. By our assumption on $Y$, each degree $\abs{y_i}$ is a multiple of $N$, and thus $E^* I_0$ is $N$-sparse, by the assumption on $E$. The distinguished triangle $Y_1 \to Y_0 \ral{p_0} I_0 \to \Si Y_1$ induces a long exact sequence on $E$-cohomology which implies that the map $I_0 \to \Si Y_1$ is injective on $E$-cohomology.  It follows that $E^*(\Si Y_1)$ is $N$-sparse as well. Repeating this process, we obtain an $\cat{I}_E$-Adams resolution of $Y$ such that for every $j$, $E^* (\Si^j Y_j)$ and $E^* (\Si^j I_j)$ are $N$-sparse.
\end{proof}

The condition on $E$ is discussed in Example \ref{ex:Compact}.

\begin{prop}\label{pr:CoherentRing}
Assume that the ring $E^* E$ is left coherent, and that $E^* Y$ is $N$-sparse and finitely presented as a left $E^*E$-module. Then $Y$ satisfies Condition~\ref{co:InjSparse}.
\end{prop}

\begin{proof}
Since $E^*Y$ is finitely generated over $E^*E$, the map $p_0 \colon Y \to I_0$ can be chosen so that $I_0 = \prod_i \Si^{m_i N} E \cong \op_i \Si^{m_i N} E$ is a finite product. 
It follows that $E^* I_0$ is $N$-sparse and finitely presented.
We have that $E^{*-1}Y_1 = \ker \left( p_0^* \colon E^*I_0 \surj E^*Y \right)$.
This is $N$-sparse, since $E^* I_0$ is, and is finitely presented over $E^*E$, since both $E^* I_0$ and $E^*Y$ are, and $E^*E$ is coherent \cite{Bourbaki98}*{\S I.2, Exercises 11--12}. %
Repeating this process, we obtain an $\cat{I}_E$-Adams resolution of $Y$ such that for every $j$, $\Si^j I_j$ is a finite product of the form $\prod_i \Si^{m_i N} E$.
\end{proof}

\subsection{Projective case}\label{ss:sparse-projective}

The main applications of Theorem~\ref{th:SparseDr} are to projective classes instead of injective classes. For future reference, we state here the dual statements of the previous subsection and adopt a notation inspired from stable homotopy theory.

\begin{nota}
Let $R$ be an object of the triangulated category $\cat{T}$. Define the \Def{homotopy} (with respect to $R$) of an object $X$ as the graded abelian group $\pi_* X$ given by $\pi_n X := \cat{T}(\Si^n R,X)$. Precomposition makes $\pi_*X$ into a right module over the graded endomorphism ring $\pi_* R$.%
\end{nota}

\begin{assum}
For the remainder of this subsection, we assume the following.
\begin{enumerate}
\item The triangulated category $\cat{T}$ has infinite %
coproducts.
\item The graded ring $\pi_* R$ is $N$-sparse for some $N \geq 2$.
\end{enumerate}
\end{assum}

Let $\cat{P}_R$ denote the stable projective class spanned by $R$, as in Example~\ref{ex:GhostSphere}. Explicitly, $\cat{P}_R$ consists of retracts of (arbitrary) coproducts $\oplus_i \Si^{n_i} R$.

\begin{condition}\label{co:ProjSparse}
$X$ admits a $\cat{P}_R$-Adams resolution $X_{\bu}$ as in Diagram~\eqref{eq:AdamsResolProj} with the property that $\pi_* (\Si^{-j} P_j)$ is $N$-sparse for each projective $P_j$.
\end{condition}

\begin{thm}\label{th:SparseDrProj}
Let $X$ and $Y$ be objects in $\cat{T}$ and consider the Adams spectral sequence abutting to $\cat{T}(X,Y)$ with respect to the projective class $\cat{P}_R$. Assume that $X$ satisfies Condition~\ref{co:ProjSparse}. Let $[y] \in E_r^{s,t}$ be a class represented by $y \in E_1^{s,t} = \cat{T}(\Si^{t-s} P_s, Y)$. Then for all $r \leq N$, the Adams differential is given, as a subset of $E_1^{s+r,t+r-1}$, by
\[
d_r [y] = \lan y, d_1, \Si^{-1} d_1, \ldots, \Si^{-(r-1)} d_1 \ran.
\]
\end{thm}

Note that we used Corollary~\ref{co:SelfDual} to ensure that the equality holds as stated, not merely up to sign.

\begin{prop}\label{pr:HomotCoproduct}
Assume that every coproduct of the form $\oplus_i \Si^{m_i N} R$ has homotopy\break $\pi_* \left( \oplus_i \Si^{m_i N} R \right)$ which is $N$-sparse. Then every object $X$ such that $\pi_* X$ is $N$-sparse also satisfies Condition~\ref{co:ProjSparse}.
\end{prop}

Recall the following terminology:

\begin{defn}
An object $X$ of $\cat{T}$ is \Def{compact} if the functor $\cat{T}(X,-)$ preserves infinite coproducts.%
\end{defn}

\begin{ex}\label{ex:Compact}
If $R$ is compact in $\cat{T}$, then $R$ satisfies the assumption of Proposition~\ref{pr:HomotCoproduct}. This follows from the isomorphism
\[
\pi_* \left( \oplus_i \Si^{m_i N} R \right) \cong \bigoplus_i \pi_* (\Si^{m_i N} R) = \bigoplus_i \Si^{m_i N} \pi_* R  
\]
and the assumption that $\pi_* R$ is $N$-sparse. The same argument works if $R$ is a retract of a coproduct of compact objects.

Dually, if $E$ is cocompact in $\cat{T}$, then $E$ satisfies the assumption of Proposition~\ref{pr:CohomProduct}. 
This holds more generally if $E$ is a retract of a product of cocompact objects.
\end{ex}

\begin{rem}
Some of the related literature deals with compactly generated triangulated categories. As noted in Remark~\ref{re:NotGenerate}, we do \emph{not} assume that the object $R$ is a generator, i.e., that the condition $\pi_* X = 0$ implies $X=0$.
\end{rem}

\begin{prop}\label{pr:CoherentRingProj}
Assume that the ring $\pi_* R$ is right coherent, and that $\pi_* X$ is $N$-sparse and finitely presented as a right $\pi_* R$-module. Then $X$ satisfies Condition~\ref{co:ProjSparse}.
\end{prop}

The following is a variant of \cite{Patchkoria12}*{Lemma 2.2.2}, where we do not assume that $R$ is a generator.
It identifies the $E_2$ term of the spectral sequence associated to the projective class $\cat{P}_R$.
The proof is straightforward.

\begin{prop}\label{pr:ExtEndoRing}
Assume that the object $R$ is compact.
\begin{enumerate}
\item Let $P$ be in the projective class $\cat{P}_R$. Then the map of abelian groups
\[ %
\cat{T}(P,Y) \to \Hom_{\pi_* R} (\pi_* P, \pi_* Y)
\] %
is an isomorphism for every object $Y$.
\item There is an isomorphism
\[
\Ext^{s}_{\cat{P}_R}(X,Y) \cong \Ext^{s}_{\pi_* R}(\pi_* X, \pi_* Y)
\]
which is natural in $X$ and $Y$. 
\end{enumerate}
\end{prop}

\subsection{Examples}\label{ss:sparse-examples}

Theorem~\ref{th:SparseDrProj} applies to modules over certain ring spectra. We describe some examples, along the lines of \cite{Patchkoria12}*{Examples 2.4.6 and 2.4.7}.

\begin{ex}\label{ex:RingSpectrum}
Let $R$ be an $A_{\infty}$ ring spectrum, and let $h\Mod{R}$ denote the homotopy category of the stable model category of (right) $R$-modules %
\cite{SchwedeS03}*{Example 2.3(ii)} 
\cite{EKMM97}*{\S III}. Then $R$ itself, the free $R$-module of rank $1$, is a compact generator for $h\Mod{R}$. The $R$-homotopy of an $R$-module spectrum $X$ is the usual homotopy of $X$, as suggested by the notation:
\[
h\Mod{R}(\Si^n R, X) \cong h\Mod{S}(S^n, X) = \pi_n X.
\]
In particular, the graded endomorphism ring $\pi_* R$ is the usual coefficient ring of $R$.

The projective class $\cat{P}_R$ is the ghost projective class \cite{Christensen98}*{\S 7.3}, generalizing  Example~\ref{ex:GhostSphere}, where $R$ was the sphere spectrum $S$. The Adams spectral sequence relative to $\cat{P}_R$ is 
the universal coefficient spectral sequence
\[
\Ext_{\pi_* R}^{s}(\Si^t \pi_* X, \pi_* Y) \Ra h\Mod{R}(\Si^{t-s} X,Y)
\]
as described in \cite{EKMM97}*{\S IV.4} and~\cite{Christensen98}*{Corollary 7.12}. %
We used Proposition~\ref{pr:ExtEndoRing} to identify the $E_2$ term.

Some $A_{\infty}$ ring spectra $R$ with sparse homotopy $\pi_* R$ are discussed in \cite{Patchkoria12}*{\S 4.3, 5.3, 6.4}. In view of Proposition~\ref{pr:ExtEndoRing}, the Adams spectral sequence in $h\Mod{R}$ collapses at the $E_2$ page if $\pi_* R$ has (right) global dimension less than $2$. %

The Johnson--Wilson spectrum $E(n)$ has coefficient ring
\[
\pi_* E(n) = \Z_{(p)}[v_1, \ldots, v_n, v_n^{-1}], \quad \abs{v_i} = 2(p^i - 1),
\]
which has global dimension $n$ and is $2(p-1)$-sparse. Hence, Theorem~\ref{th:SparseDrProj} applies in this case to the differentials $d_r$ with $r \leq 2(p-1)$, while $d_r$ is zero for $r > n$.
Likewise, connective complex $K$-theory $ku$ has coefficient ring
\[
\pi_* ku = \Z[u], \quad \abs{u} = 2,
\]
which has global dimension $2$ and is $2$-sparse.
\end{ex}

\begin{ex}\label{ex:DGA}
Let $R$ be a differential graded (\emph{dg} for short) algebra over a commutative ring $k$, and consider the category of dg $R$-modules $\dgMod{R}$. The homology $H_* X$ of a dg $R$-module is a (graded) $H_* R$-module. %
The derived category $D(R)$ is defined as the localization of $\dgMod{R}$ with respect to quasi-isomorphisms. 
The free dg $R$-module $R$ is a compact generator of $D(R)$. The $R$-homotopy of an object $X$ of $D(R)$ is its homology $\pi_* X = H_* X$. In particular, the graded endomorphism ring of $R$ in $D(R)$ is the graded $k$-algebra $H_* R$.

The Adams spectral sequence relative to $\cat{P}_R$ is 
an Eilenberg--Moore spectral sequence
\[
\Ext_{H_* R}^{s} \left( \Si^t H_* X, H_* Y \right) \Ra D(R)(\Si^{t-s} X, Y)
\]
from ordinary $\Ext$ to differential $\Ext$, as described in \cite{BarthelMR14}*{\S 8, 10}. See also \cite{KrizM95}*{\S III.4}, \cite{HoveyPS97}*{Example 10.2(b)}, 
and \cite{EilenbergM66}.
\end{ex}

\begin{rem}
Example~\ref{ex:DGA} can be viewed as a special case of Example~\ref{ex:RingSpectrum}. Letting $HR$ denote the Eilenberg--MacLane spectrum associated to $R$, the categories $\Mod{HR}$ and $\dgMod{R}$ are Quillen equivalent, by \cite{SchwedeS03}*{Example 2.4(i)} \cite{Shipley07HZ}*{Corollary 2.15}, yielding a triangulated equivalence $h\Mod{HR} \cong D(R)$. The generator $HR$ corresponds to the generator $R$ via this equivalence.
\end{rem}

\begin{ex}\label{ex:Ring}
Let $R$ be a ring, viewed as a dg algebra concentrated in degree $0$. Then Example~\ref{ex:DGA} yields the ordinary derived category $D(R)$. The graded endomorphism ring of $R$ in $D(R)$ is $H_* R$, which is $R$ concentrated in degree $0$.
This is $N$-sparse for any $N \geq 2$.  

The Adams spectral sequence relative to $\cat{P}_R$ is the hyperderived functor spectral sequence
\[
\Ext_{H_* R}^{s} \left( \Si^t H_* X, H_* Y \right) = \prod_{i \in \Z} \Ext_{R}^{s} \left( H_{i-t} X, H_{i} Y \right) \Ra D(R)(\Si^{t-s}X, Y) = \mathbf{Ext}_{R}^{s-t}(X,Y)
\]
from ordinary $\Ext$ to hyper-$\Ext$, as described in \cite{Weibel94}*{\S 5.7, 10.7}.%
\end{ex}

\appendix

\section{Computations in the stable module category of a group}\label{se:StableMod}

In this appendix, we give some computations in the stable module category
of a group algebra $kG$, where $k$ is a field and $G$ is a finite group.
These computations are used in Proposition~\ref{pr:proper-inclusion}.

Write $R$ for the group algebra $kG$.
We will work in the stable module category $\cat{T} := \StMod(R)$.
This is the category whose objects are (left) $R$-modules,
and whose morphisms from $M$ to $N$ consist of the $R$-module homomorphisms
from $M$ to $N$ modulo those that factor through a projective module.
An isomorphism in $\StMod(R)$ is called a \Def{stable equivalence},
and two $R$-modules $M$ and $N$ are stably equivalent if and only if
there are projectives $P$ and $Q$ such that $M \oplus P \cong N \oplus Q$.
The category $\StMod(R)$ is triangulated.
The suspension $\Si M$ is defined by choosing an embedding of $M$
into an injective module and taking the quotient, 
the desuspension $\Om M$ is defined by choosing a surjection from
a projective to $M$ and taking the kernel,
and these are inverse to each other because the projectives and injectives coincide.
Given a short exact sequence
\[
  0 \lra M_1 \lra M_2 \lra M_3 \lra 0 
\]
and an embedding of $M_1$ into an injective module $I$, one can choose maps
\begin{equation}\label{eq:stmod-triangle}
  \cxymatrix{
    0 \ar[r] & M_1 \ar[r] \ar@{=}[d] & M_2 \ar[r] \ar@{-->}[d] & M_3 \ar[r] \ar@{-->}[d] & 0 \\
    0 \ar[r] & M_1 \ar[r]            & I \ar[r]                & \Si M_1 \ar[r]       & 0 \\
  }
\end{equation}
making the diagram commute in $\Mod{R}$. The distinguished triangles are defined to be those triangles isomorphic in $\StMod(R)$ to one of
the form
\[
  M_1 \lra M_2 \lra M_3 \lra \Si M_1
\]
constructed in this way.

Rather than working with respect to an injective class in $\cat{T}$, we will
consider the \Def{ghost projective class} $\cat{P}$, which is generated by the trivial module $k$.
More precisely, $\cat{P}$ consists of the retracts of coproducts $\op_i \, \Si^{n_i} k$,
and the associated ideal consists of the maps which induce the zero map in Tate cohomology.
See \cite{Chebolu08}*{\S 4.2} for details.

\begin{prop}\label{pr:proper-inclusion-C4}
Let $G$ be the cyclic group $C_4 = \lan g \mid g^4 = 1 \ran$,
let $k = \F_2$, and write $R = kG$.
There exists an $R$-module $M$, an Adams resolutions of $M$ with respect to
the ghost projective class, and a map $\ka : M \to M$ such that
the inclusion 
$\lan \ka, d_1, \de \ran (\Sigma p) \subseteq \lan \ka, d_1, d_1 \ran$
from Proposition~\ref{pr:inclusions} (dualized) is proper.
\end{prop}

\begin{proof}
To produce our counterexample, we will consider the Adams spectral sequence abutting
to $\StMod(M, \Om^* M)$, where $M$ is a two-dimensional module with basis vectors
that are interchanged by $g$.

In order to make concrete computations, it will be helpful to observe that,
as a $k$-algebra, $R$ is the truncated polynomial algebra
\[
R = k[g] \big/ (g^4 - 1)
= k[g] \big/ (g-1)^4
= k[x] \big/ x^4 ,
\]
where we define $x := g-1 \in R$.
In this notation, the trivial module $k$ is $R / x$ and the module $M$ is $R / x^2$.

We will need to compute their desuspensions, which are given, as $R$-modules, by:
\begin{align*}
\Om k &= \ker \left( R \surj k \right)
= k \left\{ x, x^2, x^3 \right\}
\cong R / x^3 \\
\Om^2 k &= \ker \left( R \surj R / x^3 \right)
= k \left\{ x^3 \right\}
\cong R / x = k \\
\Om M &= \ker \left( R \surj R / x^2 \right)
= k \left\{ x^2, x^3 \right\}
\cong R / x^2 = M,
\end{align*}
where curly brackets denote the $k$-vector space with the given generating set.

In order to produce a $\cat{P}$-epic map to $M$, we need to know the maps
from suspensions of $k$ to $M$.  Since $k$ is $2$-periodic, the following
calculations give us what we need:
\begin{align*}
\cat{T}(k,M) &= \Mod{R}(k,M) / {\sim}
\cong \Mod{R}( R/x, R/x^2) / {\sim}
= k \left\{ \mu_x \right\} / {\sim}
= k \left\{ \mu_x \right\} \\
 \cat{T}(\Om k,M) &= \Mod{R}(\Om k,M) / {\sim}
= \Mod{R}( R/x^3, R/x^2) / {\sim}
= k \left\{ \mu_1, \mu_x \right\} / {\sim}
= k \left\{ \mu_1 \right\},
\end{align*}
where $f \sim g$ if $f - g$ factors through a projective,
and $\mu_r \colon R/x^m \to R/x^n$ denotes the $R$-module map given by
multiplication by $r \in R$ (when this is well-defined).
Here, we used the fact that $\mu_x \colon R/x^3 \to R/x^2$ is stably null, since it factors as
\[
R/x^3 \ral{\mu_x} R \ral{\mu_1} R/x^2.
\]
Using this, we obtain a $\cat{P}$-epic map $p := \mu_x + \mu_1 \colon k \op \Om k \to M$.
Since $p$ is surjective, its fiber is its kernel.
This kernel is generated by $(1, x)$ and is readily seen to be isomorphic to $M$.
Under the identification of $\Om M$ with $M$, 
the natural map $\Om M \to M$ (using the dual of Equation~\eqref{eq:stmod-triangle}) is $\mu_x$.
Since we are working at the prime $2$, fibre sequences and cofiber sequences agree, so
we obtain the following Adams resolution of $M$
\[
\xymatrix{
M \circar[rr]^{\mu_x} & & M \ar[dl]^{\de} \circar[rr]^{\mu_x} & & M \ar[dl]^{\de} \circar[rr]^{\mu_x} & & M \ar[dl]^{\de} \circar[r] & \cdots \\
& k \op \Om k \ar@{->>}[ul]^{p} & & k \op \Om k \ar@{->>}[ul]^{p} & & k \op \Om k \ar@{->>}[ul]^{p} , & & \\
}
\]
where $\delta = \left[ \begin{smallmatrix} \mu_1 \\ \mu_x \end{smallmatrix} \right]$,
and we have chosen to put the degree shifts on the horizontal arrows.

We will be considering the Adams spectral sequence formed by applying the
functor $\cat{T}(-, M)$.
The map $d_1 = \de p \colon k \op \Om k \to k \op \Om k$ is
$\left[ \begin{smallmatrix} 0 & \mu_1 \\ \mu_{x^2} & \mu_x \end{smallmatrix} \right]$,
which simplifies to
$\left[ \begin{smallmatrix} 0 & \mu_1 \\ \mu_{x^2} & 0 \end{smallmatrix} \right]$,
using the fact that $\mu_x \colon \Om k \to \Om k$ is stably null,
since it factors as $\Om k \ral{\mu_x} R \ral{\mu_1} \Om k$.
The stable maps $k \op \Om k \to M$ are of the form $[\, a \mu_x \; b \mu_1 \,]$
for $a$ and $b$ in $k$,
and all composites $[\, a \mu_x \; b \mu_1 \,] \circ d_1$ are stably null.
Using this twice for $d_1$'s in different positions, one sees that
if $\ka : k \op \Om k \to M$ is any map,
then $d_2[\ka]$ is defined and has no indeterminacy.

Now consider $\lan \ka, d_1, \de \ran (\Si p)$.
One part of the indeterminacy here consists of maps of the form
$f \Si(\de) \Si(p) = f \Si(d_1)$, for $f : \Si(k \op \Om k) \to M$.
As above, all such composites are zero.
The other part of the indeterminacy consists of maps of the form
$\ka f \Si(p)$, for $f : \Si M \to k \op \Om k$, and again,
one can show that all such composites are zero.
So $\lan \ka, d_1, \de \ran (\Si p)$ has no indeterminacy
 and therefore equals $d_2[\ka]$.

Finally, consider $\lan \ka, d_1, d_1 \ran$.
The part of the indeterminacy involving $d_1$ is again zero.
The other part consists of all composites $\ka f$,
for $f : \Si(k \op \Om k) \to k \op \Om k$.
Since there is an isomorphism $\Si(k \op \Om k) \to k \op \Om k$,
this indeterminacy is non-zero if and only if $\ka$ is non-zero.

Since non-zero maps $\ka : k \op \Om k \to M$ exist, we conclude
that the containment
\[
\lan \ka, d_1, \de \ran (\Si p) \subseteq \lan \ka, d_1, d_1 \ran
\]
can be proper.
\end{proof}

\begin{rem}
If in the proof above we take $\ka$ to be the map
$[\, \mu_x \; 0 \,] : k \op \Om k \to M$, then using the same techniques
one can show that
\begin{align*}
  \lan \ka, d_1, \de \ran &= \{ 1_M \} , \\
  \lan \ka, d_1, \de \ran (\Si p) &= \{ \Si p \} = d_2[\ka] = \left\{ \begin{bsmallmatrix} \mu_1 & \mu_x \end{bsmallmatrix} \right\}, \\
\intertext{and}
   \lan \ka, d_1, d_1 \ran &= \left\{ \begin{bsmallmatrix} \mu_1 & b \mu_x \end{bsmallmatrix} \mid b \in \F_2 \right\}
\end{align*}
as subsets of $\cat{T}(\Om k \op k, M) \cong \cat{T} \left( \Si (k \op \Om k), M \right)$, where we identify $M$ with $\Om M$ and $\Si M$, as before.
\end{rem}

\begin{rem}
Theorem~\ref{th:SparseDrProj} does not apply to the example in Proposition~\ref{pr:proper-inclusion-C4}. Indeed, the graded endomorphism ring of $k$ in $\StMod(kG)$ is the Tate cohomology ring $\tilde{H}^n(G;k) = \StMod(kG)(\Om^n k, k)$ 
\cite{Carlson96}*{\S 6}.
This ring is not sparse, as we have $\tilde{H}^{-1}(G;k) \neq 0$.
\end{rem}

\begin{ex}\label{ex:negative}
The following example illustrates the fact that a Toda bracket need not be equal to its own negative, as noted in Remark \ref{re:3-fold-negation}.

Consider the ground field $k = \F_3$ and the group algebra $R = kC_3 \cong k[x]/x^3$,
where we denote $x = g-1 \in R$ for $g \in C_3$ a generator.
Consider the $R$-modules $k = R/x$ and $M = R/x^2$. Let us compute the Toda bracket of the diagram
\[
\xymatrix{
M \ar[r]^-{\mu_1} & k \ar[r]^-{\mu_x} & M \ar[r]^-{\mu_1} & k \\
}
\]
in the triangulated category $\cat{T} = \StMod(R)$. We will use appropriate isomorphisms $\Si k \cong M$ and $\Si M \cong k$, and in particular compute the Toda bracket as a subset of $\cat{T}(k,k) \cong \cat{T}(\Si M, k)$. Via these isomorphisms, the suspension $\Si \mu_1 \colon \Si M \to \Si k$ equals $\mu_x \colon k \to M$. Consider the commutative diagram in $\cat{T}$
\[
\xymatrix{
& & k \ar@{-->}[d]_{\Si \al} \ar[r]^-{-\mu_x} & M \ar@{=}[d] \\
k \ar[r]^-{\mu_x} & M \ar@{=}[d] \ar[r]^-{\mu_1} & k \ar@{-->}[d]^{\be} \ar[r]^-{\mu_x} & M \\
& M \ar[r]^-{\mu_1} & k & \\ 
}
\]
where the middle row is distinguished. The only choices for the dotted arrows are $\Si \al = -1_k$ and $\be = 1_k$, from which we conclude
\[
\lan \mu_1, \mu_x, \mu_1 \ran_{\fc} = \{ -1_k \} \subset \cat{T}(k,k).
\]
\end{ex}

\section{\texorpdfstring{$3$}{3}-fold Toda brackets determine the triangulated structure}\label{se:Heller}

Heller proved the following theorem in \cite{Heller68}*{Theorem 13.2}. We present an arguably simpler proof here. The argument was kindly provided by Fernando Muro.

\begin{thm}
Let $\cat{T}$ be a triangulated category. Then the diagram $X \ral{f} Y \ral{g} Z \ral{h} \Si X$ in $\cat{T}$ is a distinguished triangle if and only if the following two conditions hold.
\begin{enumerate}
\item The sequence of abelian groups
\[
\xymatrix{
\cat{T}(A,\Si^{-1} Z) \ar[r]^-{(\Si^{-1} h)_*} & \cat{T}(A,X) \ar[r]^-{f_*} & \cat{T}(A,Y) \ar[r]^-{g_*} & \cat{T}(A,Z) \ar[r]^-{h_*} & \cat{T}(A,\Si X) \\
}
\]
is exact for every object $A$ of $\cat{T}$.
\item The Toda bracket $\lan h, g, f \ran \subseteq \cat{T}(\Si X, \Si X)$ contains the identity map $1_{\Si X}$.
\end{enumerate}
\end{thm}

\begin{proof} ($\Ra$) A distinguished triangle satisfies the first condition. For the second condition, consider the commutative diagram 
\[
\xymatrix{
X \ar@{=}[d] \ar[r]^-{f} & Y \ar@{=}[d] \ar[r]^-{g} & Z \ar[d]^{1_{Z}} \ar[r]^-{h} & \Si X \ar[d]^{1_{\Si X}} \\
X \ar[r]^-{f} & Y \ar[r]^-{g} & Z \ar[r]^-{h} & \Si X. \\
}
\]
Since the top row is distinguished, this diagram exhibits the membership $1_{\Si X} \in \lan h, g, f \ran$.

($\Leftarrow$) Assume that $1_{\Si X} \in \lan h, g, f \ran$ holds. By definition of the Toda bracket, there exists a map $\phy \colon C_f \to Z$ making the diagram 
\[
\xymatrix{
X \ar[r]^-{f} \ar@{=}[d] & Y \ar@{=}[d] \ar[r]^-{q} & C_f \ar[r]^-{\io} \ar[d]^{\phy} & \Si X \ar[d]^{1_{\Si X}} \\
X \ar[r]^-{f} & Y \ar[r]^-{g} & Z \ar[r]^-{h} & \Si X \\
}
\]
commute, where the top row is distinguished. To show that the bottom row is distinguished, it suffices to show that $\phy \colon C_f \to Z$ is an isomorphism. By the Yoneda lemma, it suffices to show that $\phy_* \colon \cat{T}(A,C_f) \to \cat{T}(A,Z)$ is an isomorphism for every object $A$ of $\cat{T}$.

Consider the diagram
\begin{equation} \label{eq:DiagCofibSeq}
\xymatrix{
X \ar[r]^-{f} \ar@{=}[d] & Y \ar@{=}[d] \ar[r]^-{q} & C_f \ar[r]^-{\io} \ar[d]^{\phy} & \Si X \ar[d]^{1_{\Si X}} \ar[r]^{- \Si f} & \Si Y \ar[d]^{1_{\Si Y}} \\
X \ar[r]^-{f} & Y \ar[r]^-{g} & Z \ar[r]^-{h} & \Si X \ar[r]^{- \Si f} & \Si Y \\
}
\end{equation}
Applying $\cat{T}(A,-)$ yields the diagram of abelian groups
\[
\xymatrix{
\cat{T}(A,X) \ar[r]^-{f_*} \ar@{=}[d] & \cat{T}(A,Y) \ar@{=}[d] \ar[r]^-{q_*} & \cat{T}(A,C_f) \ar[r]^-{\io_*} \ar[d]^{\phy_*} & \cat{T}(A,\Si X) \ar[d]^{1} \ar[r]^{(- \Si f)_*} & \cat{T}(A,\Si Y) \ar[d]^{1} \\
\cat{T}(A,X) \ar[r]^-{f_*} & \cat{T}(A,Y) \ar[r]^-{g_*} & \cat{T}(A,Z) \ar[r]^-{h_*} & \cat{T}(A, \Si X) \ar[r]^{(- \Si f)_*} & \cat{T}(A, \Si Y). \\
}
\]
The top row is exact, since the top row of \eqref{eq:DiagCofibSeq} is a cofiber sequence, and the bottom row is exact, using the first condition. By the 5-lemma, $\phy_*$ is an isomorphism.
\end{proof}

\begin{rem}
Some remarks about the first condition.
\begin{enumerate}
\item It implies $gf = g_* f_*(1_X) = 0$ and $hg = h_* g_*(1_Y) = 0$.
\item It is equivalent to the exactness of the long sequence (infinite in both directions)
\[
\xymatrix@C+4pt{
\hspace*{1em}\cdots \ar[r] 
& \cat{T}(A,\Si^n X) \ar[r]^-{(\Si^n f)_*} & \cat{T}(A, \Si^n Y) \ar[r]^-{(\Si^n g)_*} & \cat{T}(A, \Si^n Z) \ar[r]^-{(\Si^n h)_*} 
& \cat{T}(A,\Si^{n+1} X) \ar[r] 
& \cdots \\
}
\]
for every object $A$ of $\cat{T}$.
\item It is a weaker version of what is sometimes called a \emph{pre-triangle} \cite{Neeman01}*{\S 1.1}. Indeed, the condition states that the sequence
\[
\xymatrix @C=\bigcol {
H(\Si^{-1} Z) \ar[r]^-{H(\Si^{-1} h)} & H(X) \ar[r]^-{H(f)} & H(Y) \ar[r]^-{H(g)} & H(Z) \ar[r]^-{H(h)} & H(\Si X) \\
}
\]
is exact for every decent homological functor $H \colon \cat{T} \to \Ab$ of the form $H = \cat{T}(A,-)$.
\end{enumerate}
\end{rem}

\begin{cor}
Given the suspension functor $\Si \colon \cat{T} \to \cat{T}$, $3$-fold Toda brackets in $\cat{T}$ determine the triangulated structure. In particular, $3$-fold Toda brackets determine the higher Toda brackets, via the triangulation.
\end{cor}

\begin{rem}
It is unclear to us if the higher Toda brackets can be expressed directly in terms of three-fold brackets.
\end{rem}

\vspace*{10pt}

\end{document}